\documentclass[12pt]{article}

\usepackage{amsmath,amssymb,amsfonts,amsthm}
\usepackage{graphics}
\usepackage{epsfig}
 \usepackage{amsmath}
\usepackage{amsfonts}
\usepackage{color}
\allowdisplaybreaks
\font\elevensf=cmss10 scaled\magstephalf

 \newtheorem{theorem} {{\elevensf THEOREM}}[section]
 \newtheorem{proposition} {{\elevensf PROPOSITION}}[section]

\newtheorem{corollary} {{\elevensf COROLLARY}}[section]

\newtheorem{remark} {{\elevensf REMARK}}[section]

\renewcommand\qed{$\blacksquare$}

\def\CC{{\rm \kern.24em \vrule width.02em height1.4ex depth-.05ex \kern-.26emC}}
\def\dive{{\rm div} \,}

\def\TagOnRight

\def\AA{{it I} \hskip-3pt{\tt A}}

\def\QQ{\rlap {\raise 0.4ex \hbox{$\scriptscriptstyle |$}} {\hskip -0.1em Q}}

\newcommand{\lb}{\left(}
\newcommand{\rb}{\right)}
\catcode`\@=11

\def\theequation{\@arabic{\c@section}.\@arabic{\c@equation}}

\begin{document}
\baselineskip 14pt
\parindent.4in
\catcode`\@=11 

\begin{center}

{\Huge \bf Homogenization of Stokes System using  Bloch Waves } \\[5mm]

{\bf Gr\'{e}goire ALLAIRE $^{a}$, Tuhin GHOSH $^{b}$ and Muthusamy VANNINATHAN $^{b}$ } \end{center} 
\noindent
\textit{a: Centre de Mathématiques Appliqu\'{e}es, Ecole Polytechnique, CNRS, Universit\'e Paris-Saclay, 91128 Palaiseau, France}.\\ 
\textit{b: Centre for Applicable Matematics, Tata Institute of Fundamental Research, India}.\\
\textit{email : gregoire.allaire@polytechnique.fr\ ,\ tuhin@math.tifrbng.res.in\ ,\ vanni@math.tifrbng.res.in }
\begin{abstract}
\noindent
In this work, we study the Bloch wave homogenization for the Stokes system with periodic viscosity coefficient. In particular, we obtain the spectral interpretation of the homogenized tensor. The presence of the incompressibility constraint in the model raises new issues linking the homogenized tensor and the Bloch spectral data. 
The main difficulty is a lack of smoothness for the bottom of the Bloch spectrum, a phenomenon which is not present in the case of the elasticity system. This issue is solved in the present work, 
completing the homogenization process of the Stokes system via the Bloch wave method. 
\end{abstract}


\section{ Introduction and Main Result}\label{sec1} 
We consider the Stokes system in which the viscosity is a  periodically  
varying function of the space variable with small period $\epsilon > 0$. 
Many physical phenomena (boiling flows, porous media, oil reservoirs, etc.) lead to mixture of fluids with different viscosities. 
For incompressible slow or creeping flows, such a situation is modeled by the system \eqref{eq1} for a  Stokesian fluid with variable viscosity which is further assumed to be a periodic function. 
From the point of view of application, it is difficult to realize such a periodic distribution of droplets of one fluid in another without deforming the periodic structure, and \eqref{eq1} may seem as too much of an idealized system. 
Therefore, we also treat another model, which is a variant of the Stokes system and is physically more relevant. 
Namely, we consider the so-called incompressible elasticity system \eqref{sop} which corresponds to a mixture 
of incompressible elastic phases in a composite material (this situation is quite common for rubber or elastomers). 

We introduce now our first model. 
Assuming that the viscosity is a periodic function, the goal is to capture the effective viscosity of the mixture.
To write down the model we start with a $1$-periodic function $\mu= \mu (y) \in   L^\infty  (\mathbb{T}^d)$  or equivalently, a $Y$-periodic function where $Y= ]0, 1[^d. $  Here  $ \mathbb{T}^d $ is the  unit torus in $\mathbb{R}^d.$ We assume $\mu (y) \geq  \mu_0 >0$ a.e in 
$\mathbb{T}^d. $  Denote by  $\mu^\epsilon= \mu^\epsilon(x) = \mu   \left(\frac{x}{\epsilon} \right)$ the  corresponding scaled function which is  $\epsilon$-periodic.  With  $f =  f (x) \in L^2(\Omega)^d$ representing  external force,  we consider the Stokes system in a  bounded 
smooth connected domain $\Omega   \subset  \mathbb{R}^d$, with no-slip boundary condition :   
 
\begin{equation} \left. \begin{array}{rllllll}
-\nabla  \cdot      (\mu^\epsilon \nabla u^\epsilon) + \nabla p^\epsilon &=& f \mbox{  in } \Omega, \\[2mm]
 \nabla \cdot  u^\epsilon &=&  0 \mbox{   in } \Omega,  \\[2mm]
 u^\epsilon   &=& 0 \mbox{ on } \partial\Omega. \\[2mm]
 \end{array} \right\} \label{eq1}  
\end{equation}
As usual,    $u^\epsilon$ and $p^\epsilon$ represent respectively the  velocity and  pressure fields of the  fluid. 
Well-posedness theory of   (\ref{eq1})  is  classical  \cite{GR}.
We recall some of its elements. To write down the weak formulation, we introduce the spaces 

\begin{equation} 
V = \left\{ v \in  H_0^1 (\Omega)^d; \quad \nabla \cdot v =0 \mbox{ in  }\Omega \right\}, \label{eq3}
\end{equation}

\begin{equation}
 H = \left\{ v \in  L^2 (\Omega)^d; \quad v\cdot \nu =0 \mbox{ on }\partial\Omega, \quad\mbox{and}\quad\nabla \cdot v =0 \mbox{ in  }\Omega  \right\}. \label{eq4} 
\end{equation}

\noindent
Here $\nu$ denotes unit outward normal to $\partial \Omega.$ Multiplying (\ref{eq1}) by $v \in  V$ gives  the following problem for $u^\epsilon$ which does not  
involve  $p^\epsilon:$   Find $u^\epsilon \in V$ satisfying  
\begin{equation} 
\int\limits_{\Omega} \mu^\epsilon \nabla u^\epsilon \cdot   \nabla v  =  
\int\limits_{\Omega}  f \cdot  v \quad \forall  \quad   v \in V.  \label{eq5} 
\end{equation}

\noindent Lax-Milgram Lemma ensures existence and uniqueness of a solution $u^\epsilon\in V$ for (\ref{eq5}). 
To get the pressure field one applies  de Rham's  Theorem in the following form  \cite{GR}:  

\begin{equation}\label{eq6}  
V^\perp = \left\{ w\in H^{-1}(\Omega) ;\  \langle w,v \rangle_{H^{-1}(\Omega),H^1_0(\Omega)} = 0, \forall v\in V\right\} = \left\{ \nabla p ; \ \  p \in L^2 (\Omega) \right\} ,	  
\end{equation}
which implies that the pressure $p^{\epsilon}$ in (\ref{eq1}) belongs to $L^2(\Omega)$. 
Since $\Omega$ is a connected set, the pressure is defined up to an additive constant. 
To guarantee the uniqueness of the pressure, we seek $p$ in the space $L^2_0(\Omega)= \{f\in L^2(\Omega) : \int_\Omega f = 0\}$ with $L^2$ norm. 
Moreover, by using Poincar\'{e} inequality and inf-sup inequality \cite{GR}, one shows that the solution $(u^{\epsilon},p^{\epsilon})\in (H^1_{0}(\Omega))^d\times L^2_0(\Omega) $ of \eqref{eq1} are uniformly bounded, namely there exists a constant $C$, independent of $\epsilon$, such that  
\begin{equation}\label{zz2}
||u^\epsilon||_{(H^1_{0}(\Omega))^d} + ||p^\epsilon||_{L^2(\Omega)} \leq C ||f||_{(L^2(\Omega))^d}.
\end{equation}
We are interested here  in the homogenization limit of (\ref{eq1}),  that is the  asymptotic limit
of the solution  $(u^\epsilon, p^\epsilon)$ as $ \epsilon \rightarrow  0.$    
This problem is very classical and its solution by means of a combination of two-scale asymptotic expansions 
and the method of oscillating test functions was provided in various references, including \cite{BLP}, 
\cite{hornung}, \cite{SP}. We recall their main results and follow the notations of \cite{BLP} (cf. chapter I, section 10). 
The {\it homogenized tensor} $(A^{*})_{\alpha \beta}^{kl}$, 
which represents \lq\lq effective viscosity\rq\rq, is defined by   

\begin{equation}\label{eq7} 
 (A^{*})^{kl}_{\alpha \beta}  = \frac{1}{|\mathbb{T}^d|} \int\limits_{\mathbb{T}^d} \mu (y )   \nabla (\chi^{k}_\alpha +y_\alpha e_k) 
 :  \nabla (\chi^{l}_\beta +y_\beta e_l)\ dy ,    
\end{equation}
\noindent in which figure the cell  test functions $\{\chi^{k}_{\alpha};\ \alpha, k =1 \ldots d\}$ solutions of the following problem in the torus 
$\mathbb{T}^d$: 

\begin{equation}\label{eq8}  \left. \begin{array}{rllllll}
 -\nabla \cdot (\mu \nabla (\chi^{k}_\alpha +y_\alpha e_k)) + \nabla \pi^{k}_\alpha &=& 0  \mbox{ in  }\mathbb{T}^d \\[2mm]
 \nabla \cdot \chi^{k}_\alpha &=& 0 \mbox{ in } \mathbb{T}^d \\[2mm]
 (\chi^{k}_\alpha, \pi^{k}_\alpha) &&\mbox{is } Y-\mbox{periodic. } \end{array} \right\} 	  
\end{equation}
We impose $\int_{\mathbb{T}^d} \chi^{k}_\alpha\ dy = \int_{\mathbb{T}^d} \pi^{k}_\alpha\ dy = 0$ to obtain uniqueness of the solutions. 
It is easy to see that the above homogenized tensor possesses the following ``simple" symmetry, for any indices $1\leq\alpha,\beta,k,l \leq d$,
\begin{equation}\label{eq9A}
(A^{*})^{kl}_{\alpha \beta} = (A^{*})^{lk}_{\beta \alpha} ,
\end{equation}
which corresponds to the fact that the fourth-order tensor $A^{*}$ is a symmetric linear map from the set of all 
matrices (or second-order tensors) into itself. Since we follow the notations of \cite{BLP}, the simple symmetry 
\eqref{eq9A} seems a bit awkward since it mixes Latin and Greek indices but it is just the usual symmetry for 
a pair of indices $(k,\alpha)$ and $(l,\beta)$ in a fourth-order tensor. In other words, \eqref{eq9A} holds 
for a simultaneous permutation of $k,l$ and $\alpha,\beta$. 

\begin{theorem}\label{thm1.1} 
The homogenized limit of the problem (\ref{eq1}) is 

\begin{equation}\label{eq9} \left. \begin{array}{rllllll}
-\frac{\partial}{\partial x_\beta} \left((A^{*})^{kl}_{\alpha \beta} \frac{\partial u_k}{\partial x_\alpha} \right) + \frac{\partial  p}{\partial x_l }&=&  f_l 
\mbox{  in } \Omega , \mbox{ for } l=1,2,...,d,\\[2mm]
 \nabla \cdot  u &=& 0 \mbox{   in } \Omega , \\[2mm]
 u  &=& 0 \mbox{ on }\partial\Omega .
\end{array} \right\} 
\end{equation}

\noindent More precisely, we have the convergence of solutions: 

$$ 
(u^\epsilon, p^\epsilon) \rightharpoonup  (u, p)  \mbox{  in  } H^1_{0} (\Omega) \times 
 {L^2_0(\Omega)} \mbox{ weak. }  
$$
\hfill\qed\end{theorem} 

\noindent
Note that the simple symmetry \eqref{eq9A} does not imply that $A^{*}$ is symmetric in $k,l$ or in $\alpha,\beta$. 
However, in the homogenized equation \eqref{eq9}, since $A^{*}$ is constant, only its symmetric version, obtained 
by symmetrizing in both $k,l$ and $\alpha,\beta$, plays a role. 

Let us next consider the second model of incompressible elasticity :
\begin{equation}\label{sop}  \left. \begin{array}{rllllll} 
 -\nabla  \cdot      (\mu^\epsilon E(u_s^\epsilon))  + \nabla p_s^\epsilon &=& f \mbox{  in } \Omega, \\[2mm]
 \nabla \cdot  u_s^\epsilon &=&  0 \mbox{   in } \Omega,  \\[2mm]
 u_s^\epsilon  &=& 0   \mbox{ on }\partial\Omega. \\[2mm]
 \end{array} \right\}  
\end{equation}
Here the strain rate tensor is given by 
$$ 
E (v) =  \frac{1}{2}\lb \nabla v + \nabla^t  v \rb \quad \mbox{ namely } \quad 
E_{kl} (v) = \frac{1}{2} \left(\frac{\partial  v_k}{\partial  x_l}  + \frac{\partial v_l}{\partial x_k} \right). 
$$
As before, there exists a unique solution $(u_s^{\epsilon},p_s^{\epsilon})$ of the above problem \eqref{sop} in 
$(H^1_{0}(\Omega))^d\times L^2_0(\Omega)$ and using Korn's inequality and the inf-sup inequality, the following uniform bound can be proved :
\begin{equation}\label{zz3}
||u_s^\epsilon||_{(H^1_{0}(\Omega))^d} + ||p_s^\epsilon||_{L^2(\Omega)} \leq C ||f||_{(L^2(\Omega))^d} ,
\end{equation}
where the constant $C$ does not depend on $\epsilon$. 
Here the {\it homogenized tensor} ${(A^{*}_s)}_{\alpha \beta} ^{kl}$ is given by 
\begin{equation}\label{zz1}
 {(A^{*}_s)}^{kl}_{\alpha \beta}  = \frac{1}{|\mathbb{T}^d|} \int\limits_{\mathbb{T}^d} \mu (y )   E (\widetilde{\chi}^{k}_\alpha +y_\alpha e_k) : E (\widetilde{\chi}^{l}_\beta +y_\beta e_l ) 	
\end{equation}
\noindent where the cell test functions $\widetilde{\chi}^{k}_{\alpha}$ are now solutions in the torus $\mathbb{T}^d$ of 

\begin{equation} \left. \begin{array}{rllllll}
 -\nabla \cdot (\mu E (\widetilde{\chi}^{k}_\alpha +y_\alpha e_k)) + \nabla \widetilde{\pi}^{k}_\alpha &=& 0  \mbox{ in  }\mathbb{T}^d \\[2mm]
 \nabla \cdot \widetilde{\chi}^{k}_\alpha &=& 0 \mbox{ in } \mathbb{T}^d \\[2mm]
 (\widetilde{\chi}^{k}_\alpha, \widetilde{\pi}^{k}_\alpha) && \mbox{is } Y- \mbox{ periodic } \end{array} \right\} 
\end{equation}
\noindent We impose $\int_{\mathbb{T}^d} \widetilde{\chi}^{k}_\alpha = \int_{\mathbb{T}^d} \widetilde{\pi}^{k}_\alpha =0$.  It is known \cite{CHK} that the above homogenized tensor possesses the following \lq \lq full" symmetry, 
for any indices $1\leq\alpha,\beta,k,l \leq d$,
\begin{equation}\label{eq9sA}
{(A_s^{*})}^{kl}_{\alpha \beta} = {(A_s^{*})}^{\alpha l}_{k \beta} = {(A_s^{*})}^{k \beta}_{\alpha l} = {(A_s^{*})}^{lk}_{\beta \alpha}, 
\end{equation}
which corresponds to the fact that the fourth-order tensor $A^{*}_s$ is a symmetric linear map from the set of all 
symmetric matrices into itself (the conditions \eqref{eq9sA} are the usual symmetry conditions for Hooke's laws in 
linearized elasticity). 
The homogenization limit of the problem (\ref{sop}) is again of the form \eqref{eq9} with $A^{*}_s$ replacing $A^{*}$. 

The first goal of this paper is to give an alternate proof of Theorem \ref{thm1.1} using the Bloch Wave Method 
instead of two-scale asymptotic expansions and the method of oscillating test functions. The notion of Bloch 
waves is well-known in physics and mathematics \cite{BLP}, \cite{CPV}, \cite{RS}, \cite{wilcox}. Bloch waves 
are eigenfunctions of a family of \lq \lq shifted \rq\rq spectral problems in the unit cell $Y$ for the corresponding 
differential operator. Its link with homogenization theory was first explored 
in \cite{BLP}, \cite{CV}, \cite{MB}, \cite{SaSy}. The key point is that the homogenized 
operator can be defined in terms of differential properties of the bottom of the Bloch spectrum. 
The second goal of this paper is to explore this issue which is especially delicate in the case of Stokes 
equations. Indeed, it was discovered in \cite{ACFO} that the Bloch spectrum for the Stokes equations is not regular enough at the origin because of the incompressibility constraint. Therefore, its differential 
properties are all the more intricate to establish. Here we complete the task started in \cite{ACFO} 
and in particular we prove a conjecture of \cite{ACFO} on the homogenization of the Stokes system \eqref{eq1}. 
Since the treatment of the incompressible elasticity system \eqref{sop} is almost analogous to that of \eqref{eq1}, 
we focus on \eqref{eq1} and we content ourselves in highlighting the main differences for \eqref{sop} throughout the sequel. 

The Bloch wave method for scalar equations and systems without differential constraints (like the incompressibility 
condition) was studied in \cite{COV,CV,GV,GV2}. 
In such cases, this approach gives a spectral representation of the homogenized tensor $A^{*}=(A^{*})^{kl}_{\alpha\beta}$ in terms of the lowest energy Bloch waves and their behaviour for small momenta (what we call the bottom of the spectrum).
For instance, the homogenized matrix in the scalar case was found to be equal to one - half of the Hessian of the ground energy 
(or first eigenvalue) at zero momentum. For a system, several bottom eigenvalues play a role and they are merely directionally differentiable by lack of simplicity. In the present case of the Stokes system, the situation is  more complicated. The main characteristic of the Stokes system is the  presence of the differential constraint expressing incompressibility of the fluid. One of its effects is that the Bloch energy levels are 
degenerate and the corresponding eigenfunctions are discontinuous at zero momentum. Even though  energy levels are  continuous at zero momentum, the second order derivatives are not (cf. Theorem \ref{thm3.1}).
Thus, we cannot really make sense of the eigenvalue Hessian at zero momentum. Further, it is not clear if the homogenized tensor can be fully recovered  from the Bloch spectral data.
In fact, this issue  is  left open in \cite{ACFO}. In the non-self adjoint case treated in \cite{GV2}, only the symmetric part of the homogenized matrix is determined by Bloch 
spectral data and this is enough to determine the homogenized operator uniquely. Combining all these difficulties, the homogenization of Stokes system using Bloch waves is an interesting issue  which is not a direct extension of previous results. 
Our work, roughly speaking, shows that Bloch spectral data does not determine 
the homogenized tensor uniquely, but determines the homogenized operator uniquely. This is in sharp contrast with the linear elasticity system treated in  \cite{GV} in which the homogenized tensor was uniquely determined from Bloch spectral data. We see thus the effect of differential constraints (the incompressibility condition in the case of Stokes equations) on the homogenization process via Bloch wave method. For further discussion on this point,  see Section \ref{sec4}. Bloch wave method of homogenization presented in Section \ref{sec5} consists of localizing \eqref{eq1}, taking its Bloch transform and passing to the limit to get the localized version of homogenized system in the Fourier space. Passage to the limit in the Bloch method is straight forward, though arguments are long. We do not run into the classical difficulty of having a product of two weakly convergent sequences. In fact, we use the Taylor approximation of Bloch spectral elements which gives strongly convergent sequences. This is one of the known features of the method. The required homogenized system is obtained by making a passage to the physical space from the Fourier space. Extracting macro constitutive relation and macro balance equation from the localized homogenized equation in the Fourier space turns out to be not very straight forward because of differential constraints.  
       
The plan of this paper is as follows. In section \ref{sec2}, we recall from \cite{ACFO} the properties of Bloch waves associated with the Stokes operator. It turns out that the Bloch waves and their energies can be chosen to be directionally regular, upon modifying the spectral cell problem at zero momentum. Bloch transform using eigenfunctions lying at the bottom of the spectrum is also introduced in this section. Its asymptotic behaviour for low momenta is also described. Next, Section \ref{sec3} is devoted to the computation of directional derivatives of Bloch spectral data. Even though these results are essentially borrowed from \cite{ACFO}, some new ones are also included because of their need in the sequel. In particular we derive the so-called propagation relation linking the homogenized tensor $A^{*}$ with Bloch spectral data, and the extent to which it determines homogenized tensor is studied in Section \ref{sec4}.  Using this information, we prove Theorem \ref{thm1.1}  in Section \ref{sec5} following the Bloch wave homogenization method.

\section{Bloch waves}\label{sec2}
\setcounter{equation}{0} 

In this section, we introduce Bloch waves associated to the Stokes operator following the lead of \cite{ACFO}. 
The Bloch waves are defined by considering the shifted (or translated) eigenvalue problem in the torus 
$\mathbb{T}^d$ parametrized by elements in the dual torus which we take as $\mathbb{T}^d$ again. 
We denote by $y$ the points of the original torus and by $\eta$ the points of the dual torus.   
The spectral Bloch problem amounts to find $\lambda = \lambda (\eta) \in \mathbb{R}$, 
$\phi =\phi  (\eta)  \in  (H^1 (\mathbb{T}^d))^d$, with $\phi \neq 0$ and  $\pi = \pi (\eta) \in L^2 (\mathbb{T}^d)$, 
satisfying 

\begin{equation}\label{eq2.1} 
\left. \begin{array}{rllllll}
- D (\eta) \cdot (\mu D (\eta) \phi) + D (\eta) \pi  &=& \lambda  (\eta) \phi \mbox{ in } \mathbb{T}^d ,\\[2mm]
 D (\eta) \cdot \phi &=& 0 \mbox{ in } \mathbb{T}^d ,\\[2mm]
 (\phi, \pi)  \mbox{ is } Y &-& \mbox{ periodic,} \\[2mm]
\int\limits_{Y} |\phi|^2 dy &=&1.
\end{array} \right\}  
\end{equation}

\noindent 
The solutions of \eqref{eq2.1} are a priori complex valued, so all functional spaces are complex valued too. 
Here, we denote 
$$
D (\eta) =\nabla+ i\eta
$$ 
the shifted gradient operator, with $i$ the imaginary root $\sqrt{-1}$. 
Its action on a vector function $\phi$ yields a matrix: $(D (\eta) \phi )_{kl} = \frac{\partial \phi_l}{\partial y_k} + i \eta_k  \phi_l$ for all $k, l=1, \ldots ,d$.   
The corresponding divergence operation yields a scalar: $D (\eta) \cdot \phi  =\frac{\partial \phi_k}{\partial y_k} + i  \eta_k  \phi_k.$  Analogously, if  $\phi$ is a  matrix function 
then its shifted divergence $D (\eta) \cdot \phi$ is a vector function obtained by acting $D (\eta)$ on the column  vectors  of $\phi.$ 

The main feature of (\ref{eq2.1}) is that the state space keeps varying with $\eta$ due to the differential constraints defined by the   incompressibility of the fluid. That is why, the standard spectral theory for elliptic operators does not apply as such; it has to be modified. This is accomplished in   \cite{OZ}. Secondly, it is easily seen that when $\eta=0$, the corresponding eigenvalue $\lambda (0)$ is equal to zero and its multiplicity is $d$. In fact, we can take $e_k, k=1 \ldots d$ as eigenvectors (with corresponding eigen-pressure  being zero). Because of this degeneracy, spectral elements of  (\ref{eq2.1}) are not 
guaranteed to be  smooth at  $\eta=0.$  Lack of regularity of the Bloch spectrum at $\eta=0$ is an issue because the representation of the homogenized tensor in terms of Bloch spectral elements is then not clear.    To overcome this difficulty, the idea is to consider directional regularity as we approach $\eta=0$ \cite{GV}. Accommodating the directional limit at $\eta=0$ requires a  modification of the above  shifted problem with the addition of a new constraint and corresponding Lagrange multiplier in the equation \cite{ACFO}.  Fixing a direction $e \in \mathbb{R}^d, |e|  =1$ and taking  $\eta=\delta e,$ with $\delta>0$, we 
consider the modified problem: find  $\lambda(\delta) \in  \mathbb{R}$, $\phi(.;\delta) \in  (H^1 (\mathbb{T}^d))^d$, 
$q(.;\delta) \in L^2_0(\mathbb{T}^d)$ where $L^2_0(\mathbb{T}^d) = \{ q \in L^2 (\mathbb{T}^d); \  \int_{\mathbb{T}^d}q = 0 \}$
and  $q_0(\delta) \in  \mathbb{C}$ satisfying   
 
\begin{equation} \left. \begin{array}{rllllll}
\displaystyle{ -D (\delta e) \cdot  (\mu(y)   D (\delta e )\phi(y;\delta))  +  D  (\delta e ) q(y;\delta)  +  q_0 (\delta)  e  } &=& \displaystyle{  \lambda(\delta)  \phi(y;\delta)  \mbox{   in } \mathbb{T}^d ,} \\[2mm]
\displaystyle{  D (\delta e) \cdot \phi(y;\delta) } &=&  0  \mbox{  in }  \mathbb{T}^d , \\  
\displaystyle{ e \cdot \int\limits_{\mathbb{T}^d}  \phi(y;\delta)  dy  }&=& 0, \\
 (\phi, q)  \mbox{ is } Y &-& \mbox{ periodic,} \\[2mm]
\displaystyle{ \int\limits_{\mathbb{T}^d} |\phi(y;\delta)|^2  dy } &=& 1. 
\end{array} \right\} 	  \label{eq2.2} 
\end{equation}
Note that if $\delta \neq 0$ then the relation $e \cdot \int\limits_{\mathbb{T}^d}  \phi(.;\delta) =0$ can be 
easily obtained from $D (\delta e)\cdot\phi(.;\delta) =0$ simply  by integration.  (However, this is not the case if $\delta =0$.) 
Hence (\ref{eq2.2}) is the same as (\ref{eq2.1}) provided  $\delta \neq 0$ and $\eta = \delta e$.  
However for $\delta=0, $ (\ref{eq2.1}) is not good because the condition 
$e \cdot \int\limits_{\mathbb{T}^d}  \phi(.;\delta) =0$ is not included. 
See \cite{ACFO} on the appearance of this new constraint and the corresponding 
Lagrange multiplier $q_0(\delta)e$. 

It is natural to consider the system (\ref{eq2.2}) with $\delta$ small as a perturbation from the following one 
which corresponds to $\delta =0$. 
We fix a unit vector $\hat{\eta}\in \mathbb{S}^{d-1}$ and we consider the eigenvalue problem: 
find $\nu (\hat{\eta})\in\mathbb{R}$, $w(. ,\hat{\eta}) \in  (H^1 (\mathbb{T}^d))^d$, 
$q(.;\hat{\eta}) \in L^2_0(\mathbb{T}^d)$ and $q_0(\hat{\eta}) \in \mathbb{C}$ satisfying   
\begin{equation} \left. \begin{array}{cccc}
-\nabla  \cdot  (\mu \nabla w) + \nabla q  + q_0  \hat{\eta}  = \nu (\hat{\eta}) w  \mbox{ in } \mathbb{T}^d ,\\[2mm]
\nabla \cdot w =0  \mbox{ in } \mathbb{T}^d ,\\[2mm]
\hat{\eta} \cdot \int\limits_{Y}  w dy =0 ,\\[2mm]
(w, q) \mbox{ is } Y -\mbox{ periodic,} \\[2mm]
\int\limits_{Y} |w|^2 dy =1 .
\end{array} \right\} 	  \label{eq2.4} 
\end{equation}

Existence  of eigenvalues and eigenvectors for either (\ref{eq2.2}) or \eqref{eq2.4} is proved in 
\cite{ACFO}. Let us recall their result, by specializing to the eigenvalue $\nu(\hat{\eta}) =0$ of  
(\ref{eq2.4}). Note that $\nu(\hat{\eta})=0$  is clearly an eigenvalue of 
multiplicity   $(d-1)$ of  (\ref{eq2.4}) with corresponding eigenfunctions being constants, namely 
$q^0_{m,\hat{\eta}} =0$, $q^0_{0, m,\hat{\eta}} = 0$ and  $\phi_{m,\hat{\eta}}^0 (y)$ is a constant 
unit vector of $\mathbb{R}^d$ orthogonal to $\hat{\eta}$ for $m=1, \ldots , (d-1)$, say 
$\{\phi^0_{1,\hat{\eta}}, \ldots \phi^0_{d-1,\hat{\eta}}\}$. 
Doing perturbation analysis of the above situation, the following result was proved in \cite{ACFO}.

\begin{theorem}\label{thm2.1}
Fix $\hat{\eta}\in \mathbb{S}^{d-1}$. Consider the first $(d-1)$ eigenvalues of (\ref{eq2.2}). 
There exists $\delta_0>0$ and exactly $(d-1)$ analytic functions defined in the real interval $|\delta|\leq \delta_0$, 
$\delta \mapsto  \lb \lambda_{m,\hat{\eta}}(\delta), \  \phi_{m,\hat{\eta}}(.;\delta), \ q_{m,\hat{\eta}}(.;\delta), \ q_{0, m,\hat{\eta}}(\delta) \rb$, for $m= 1, \ldots , (d-1)$, with values in $\mathbb{R} \times (H^1 (\mathbb{T}^d))^d \times 
 L^2_0 (\mathbb{T}^d) \times \CC$, such that 
\begin{enumerate}
\item[(i)] $\displaystyle{ \lambda_{m,\hat{\eta}}(\delta) \bigg|_{\delta =0}  =  0,     \  \phi_{m,\hat{\eta}}(.;\delta)  \bigg|_{\delta =0}   = \phi^0_{m,\hat{\eta}} ,  \ q_{m,\hat{\eta}}(.;\delta)  \bigg|_{\delta =0} = \ q_{0,m,\hat{\eta}}(\delta)  \bigg|_{\delta =0} = 0, }$

\item[(ii)]   $\displaystyle{ \lb \lambda_{m,\hat{\eta}}(\delta),  \  \phi_{m,\hat{\eta}}(.;\delta),   \ q_{m,\hat{\eta}}(.;\delta) ,   \  q_{0, m,\hat{\eta} }(\delta) \ \rb }$ 
satisfies (\ref{eq2.2}).  

\item[(iii)] The set $ \displaystyle{  \left\{ \phi_{1,\hat{\eta}}(.;\delta), \ldots  .   \phi_{(d-1),\hat{\eta}}(.;\delta) \right\} } $ is  
orthonormal in $(L^2 (\mathbb{T}^d))^d. $

\item[(iv)]  For each interval $I \subset \mathbb{R}$ with $\overline{I}$ containing exactly the eigenvalue  $\nu(\hat{\eta})=0$ of  (\ref{eq2.4}) (and no other eigenvalue of  (\ref{eq2.4}) then $\left\{\lambda_{1,\hat{\eta}}(\delta) \ldots,  
\lambda_{(d-1),\hat{\eta}}(\delta)\right\}$ are the only eigenvalues of  (\ref{eq2.2})  (counting multiplicities) lying 
in the interval  $I.$   
\end{enumerate} 
\hfill\qed\end{theorem} 

The above theorem says that there are $(d-1)$ smooth curves emanating  out of the zero eigenvalue 
as $\delta$ varies in an interval $(-\delta_0,  \delta_0)$.  We call them Rellich branches. Using them, 
for $m= 1, \ldots , (d-1)$, we can define 
the corresponding $m^{th}$ Bloch transform of  $g \in \left(L^2 (\mathbb{R}^d)\right)^d $  via the expression
\begin{equation}\label{bloch-transformation1}
B_{m,\hat{\eta}}^\epsilon g  (\xi) =  \int\limits_{\mathbb{R}^d} g (x)  \cdot 
\overline{\phi_{m,\hat{\eta}}  \left(\frac{x}{\epsilon}, \delta \right)}   e^{-i x \cdot \xi}\ dx, 
\end{equation}
where $\delta  =\delta (\epsilon, \xi) = \epsilon |\xi|$ and $\hat{\eta} =  \xi/|\xi|.$ This is 
well defined provided $\epsilon$ is  sufficiently small so that $\epsilon |\xi|  \leq  \delta_0$. 
For other $\xi$, we define $B^\epsilon_{m,\hat{\eta}} g(\xi) =0.$  \\
\\
For later purposes we need the Bloch transform for $(H^{-1}(\mathbb{R}^d))^d$ elements also.
Let us consider $F \equiv  ( g^0 + \sum_{j=1}^d \frac{\partial}{\partial x_j}g^j ) \in (H^{-1}(\mathbb{R}^d))^d$, 
where $F, g^0, g^1,...,g^d$ are valued in $\mathbb{C}^d$ and $g^j\in ((L^2(\mathbb{R}^d))^d$ for $j=0,1,...,d$. 
Then we define $B^{\epsilon}_{m,\hat{\eta}} F(\xi)$ in $L^2_{loc}(\mathbb{R}^d_{\xi})$ by
\begin{equation}\begin{aligned}\label{bloch-transformation2}
B^{\epsilon}_{m,\hat{\eta}} F(\xi) :=  \int_{\mathbb{R}^d}  g^0(x)\cdot \overline{\phi_{m,\hat{\eta}}}\left(\frac{x}{\epsilon};\delta \right)e^{-ix\cdot\xi} dx &+ \int_{\mathbb{R}^d} i \sum_{j=1}^d \xi_j g^j(x)\cdot \overline{\phi_{m,\hat{\eta}}}\left(\frac{x}{\epsilon};\delta\right)e^{-ix\cdot\xi}dx \\
 &- \epsilon^{-1}\int_{\mathbb{R}^d} \sum_{j=1}^d g^j(x)\cdot\frac{\partial\overline{\phi_{m,\hat{\eta}}}}{\partial y_j}\left(\frac{x}{\epsilon};\delta\right)e^{-ix\cdot\xi} dx \,.
\end{aligned}
\end{equation}
Definition \eqref{bloch-transformation2} is independent of the representation used for $F\in (H^{-1}(\mathbb{R}^d))^d$ 
in terms of $\{g^j ,\ j=0,...,d\}$ and is consistent with 
the previous definition \eqref{bloch-transformation1} whenever $F\in (L^2(\mathbb{R}^d))^d$. 
\begin{remark}
Due to the property  $\nabla\cdot(e^{ix\cdot\xi}\phi^\epsilon_{m,\hat{\eta}})=0$ in $\mathbb{R}^d$, we see from \eqref{bloch-transformation2} that
\begin{equation}\label{invariant}
 B^\epsilon_{m,\hat{\eta}}(F +\nabla \psi)(\xi) = B^\epsilon_{m,\hat{\eta}}F(\xi), \mbox{ for all }\psi\in L^2(\mathbb{R}^d). 
\end{equation}
In fact, by considering $ \nabla \psi = g^0 + \sum_{j=1}^d \frac{\partial}{\partial x_j}g^j\in (H^{-1}(\mathbb{R}^d))^d$ then we can take  $g_0 = 0$ and $g_j = \psi e_j$ $j=1,\ldots,d$, in \eqref{bloch-transformation2} to obtain \eqref{invariant}. 
That is, Bloch transform of gradient field is zero. Therefore the kernel of the Bloch transform $B^\epsilon_{m,\hat{\eta}} : L^2(\mathbb{R}^d)^d \mapsto L^2(\mathbb{R}^d)$ contains the closed subspace $\{\nabla \psi  : \psi\in H^1(\mathbb{R}^d)\}$ for each $m=1,\ldots,d-1$. Roughly speaking since Bloch waves satisfy incompressibility condition the Bloch transform on gradient field vanish. Thus we may anticipate that the pressure effects may not be captured in the Bloch method. This impression is not correct. Indeed, as shown Section \ref{sec5}, by means of localization via a cut-off function, we manage to keep the pressure term.
\hfill\qed\end{remark}
\noindent
Our  next result is concerned with the asymptotic behavior of these Bloch transforms as $\epsilon \rightarrow 0$. 
Since $\phi_{m,\hat{\eta}} (y; 0)$ is a fixed unit vector $(=\phi^0_{m,\hat{\eta}})$ orthogonal 
to $\hat{\eta}$ and independent of $y$ (see Theorem \ref{thm2.1}), we have  
\begin{theorem} \label{thm2.2}
Let  $g^\epsilon$ be  a sequence in $(L^2(\mathbb{R}^d))^d$ such that its support is contained  in a fixed  
compact set $K \subset \mathbb{R}^d,$ independent  of $\epsilon.$ If $g^\epsilon$ converges   weakly to 
$g$ in $(L^2 (\mathbb{R}^d))^d, $ then we have  
\begin{equation}\label{weak}
  \chi_{\epsilon^{-1}  \mathbb{T}^d}  (\xi) B^\epsilon_{m,\hat{\eta}} g^\epsilon (\xi)  
\rightharpoonup  \phi_{m,\hat{\eta}}^0 \cdot \widehat{g} (\xi), \ \  \mbox{ weakly  in } L^2_{loc} (\mathbb{R}^d_\xi) 
\mbox{ for } 1 \leq  m \leq  d-1
\end{equation}
where $\hat{g}$ denotes the Fourier transform of $g$ and we recall that $\hat{\eta}=\frac{\xi}{|\xi|}$. 
\end{theorem}
\begin{proof}
Let us remark  that  $B^\epsilon_{m,\hat{\eta}} g^\epsilon (\xi)$ is defined for  $\epsilon \leq  \frac{\delta_0}{M} $ if  
$|\xi| \leq  M. $  We can write  
$$
B^\epsilon_{m,\hat{\eta}}  g^\epsilon (\xi)  = \chi_{\epsilon^{-1}\mathbb{T}^d}  (\xi) \phi_{m,\hat{\eta}}^0 \cdot g^\epsilon (\xi)  + 
\int_K  g^\epsilon  (x) \cdot
  \left(\overline{\phi_{m,\hat{\eta}} \left(\frac{x}{\epsilon}; \delta\right) } -  \overline{\phi_{m,\hat{\eta}}  
\left(\frac{x}{\epsilon}; 0\right)} \right)  e^{-i x \cdot \xi}  dx.
$$
By using Cauchy-Schwarz, the second term on the above right hand side can be estimated by the quantity 
$$
C_K \|\phi_{m,\hat{\eta}} (y; \delta)  -  \phi_{m,\hat{\eta}}  (y; 0) \|_{(L^2 (Y))^d}
$$ 
where $C_K$ is a constant depending on $K$ but not on $\epsilon$. Recall that $\delta$ is a 
function of $(\epsilon, \xi),$ namely $\delta =\epsilon |\xi|$. This quantity is easily seen 
to converge to zero as $\epsilon \rightarrow 0$ for each fixed $\xi$ because of the directional 
continuity of $\phi_{m,\hat{\eta}}(.,\delta) \mapsto \phi^0_{m,\hat{\eta}}$ in $(L^2(\mathbb{T}^d))^d$ as $\delta\to0$. 
We merely use the continuity of the $m^{th}$ Rellich  branch at $\delta=0$ with values in 
$(L^2 (\mathbb{T}^d))^d$. 
On the other hand, thanks to our normalization, the integral on $K$ is bounded by 
a constant independent of $(\epsilon, \xi).$ The proof is completed by a simple application of the 
Dominated Convergence Theorem which guarantees that the second term on the above right hand side 
converges strongly to $0$ in $L^2_{loc} (\mathbb{R}^d_\xi)$ as $\epsilon\to0$.  \hfill \end{proof} 
\noindent
Since compactly supported elements are dense in $(L^2(\mathbb{R}^d))^d$, we have the following :
\begin{corollary}\label{corol}
In the setting of Theorem \ref{thm2.2},
if  $g^\epsilon$ be  a sequence in $(L^2(\mathbb{R}^d))^d$ such that its support is contained  in a fixed  
compact set $K \subset \mathbb{R}^d,$ independent  of $\epsilon$ and   $g^\epsilon \rightarrow g$ in $L^2(\mathbb{R}^d)^d$ then we have the following strong convergence 
\begin{equation}\label{strong} 
\chi_{\epsilon^{-1} \mathbb{T}^d}  B^\epsilon_{m,\hat{\eta}}g^\epsilon (\xi)   \rightarrow  \phi_{m,\hat{\eta}}^0  \cdot \widehat{g}, \ \  \mbox{ strongly in }  L^2_{loc} (\mathbb{R}^d_\xi) \quad \mbox{for }  1  \leq   m \leq   d-1. 
\end{equation}
\hfill\qed\end{corollary}
\noindent
We recall the classical orthogonal decomposition :
\begin{equation}
 L^2(\mathbb{R}^d)^d = \{\nabla \psi : \psi\in H^1(\mathbb{R}^d)\} \oplus \{ \phi \in L^2(\mathbb{R}^d)^d : \nabla\cdot \phi = 0 \}.
\end{equation}
Let us denote 
\begin{equation}\label{X}
X =  \{\nabla \psi : \psi\in H^1(\mathbb{R}^d)\}, \mbox{ so that, } X^{\perp} = \{ \phi \in L^2(\mathbb{R}^d)^d : \nabla\cdot \phi = 0 \} .
\end{equation}
By our choice, $\{\phi^0_{1,\hat{\eta}},\ldots,\phi^0_{d-1,\hat{\eta}},\hat{\eta}\}$ forms an orthonormal basis in $\mathbb{R}^d$, and so we can deduce the following : 
\begin{proposition}\label{proposition2.1}
If $g\in X^{\perp}$ and $\phi^0_{m,\hat{\eta}}\cdot \widehat{g} = 0 $ for all $ m=1,\ldots, d-1$, then  $g=0$.
\end{proposition}
\begin{proof}
The proof is immediate, as $\{\phi^0_{1,\hat{\eta}},\ldots,\phi^0_{d-1,\hat{\eta}}\}$ forms an orthogonal basis in $\mathbb{R}^{d-1}$ and $\phi^0_{m,\hat{\eta}}\cdot \widehat{g} = 0 $ for all $ m=1,\ldots, d-1$, so $\widehat{g}(\xi) = c(\xi)\xi$ for some scalar $c\in L^2(\mathbb{R}^d)$. Now if $c\neq 0$, it contradicts the hypothesis $g\in X^{\perp}$. Thus $c=0$. Consequently, $g= 0$.
\hfill\end{proof}
\begin{corollary}\label{corol2}
In the setting of Theorem \ref{thm2.2} and Proposition \ref{proposition2.1}
if  $g^\epsilon$ be  a sequence in $X^{\perp} \subset L^2(\mathbb{R}^d)^d$ such that its support is contained  in a fixed compact set $K \subset \mathbb{R}^d,$ independent  of $\epsilon$ and  $g^\epsilon \rightharpoonup g$ in $L^2(\mathbb{R}^d)^d$ weak, and $\phi^0_{m,\hat{\eta}}\cdot \widehat{g} = 0 $ for all $ m=1,\ldots, d-1$, then  $g=0$. 
\end{corollary}
\begin{proof}
 The proof simply follows as $X^{\perp}$ is a closed subspace of $L^2(\mathbb{R}^d)^d$, so the limit $g\in X^{\perp}$ and the result follows by applying Proposition \ref{proposition2.1}. 
\hfill\end{proof}
\begin{remark}
 Bloch waves being incompressible are transversal. Longitudinal direction is missing and it has to be added to get the full basis. Naturally, asymptotics of the Bloch transform contains information of the Fourier transform only in transversal directions. It contains no information in the longitudinal direction. Because of this features, in the homogenization limit also, there is no information in the longitudinal direction. This is however proved to be enough to complete the homogenization process because the limiting velocity field is incompressible. See Section \ref{sec5}.   
\hfill\qed\end{remark}

\section{Computation of derivatives}\label{sec3}
\setcounter{equation}{0}
In this section, we give the expressions of the derivatives (at $\delta=0$) of the Rellich branches 
$\{\phi_{m,\hat{\eta}}(y;\delta),q_{m,\hat{\eta}}(y;\delta), q_{0,m,\hat{\eta}}(\delta),\lambda_{m,\hat{\eta}}(\delta)\}$ obtained in Theorem \ref{thm2.1}. These results are essentially borrowed from \cite{ACFO} except for the second order derivative of $q_{0,m,\hat{\eta}}(\delta)$ which is new.
We differentiate, with respect to $\delta\in\mathbb{R}$, (\ref{eq2.2}) or equivalently the 
following system, fixing $m=1,\ldots,d-1$ and $\hat{\eta} =\frac{\xi}{|\xi|}\in \mathbb{S}^{d-1}$, 
\begin{equation} \left. \begin{array}{rllllll}
\displaystyle{ -D (\delta \hat{\eta}) \cdot  (\mu(y)   D (\delta \hat{\eta} )\phi_{m,\hat{\eta}}(y;\delta))  +  D  (\delta \hat{\eta} ) q_{m,\hat{\eta}}(y;\delta)  +  q_{0,m,\hat{\eta}} (\delta)  \hat{\eta}  } &=& \displaystyle{  \lambda_{m,\hat{\eta}}(\delta)  \phi_{m,\hat{\eta}}(y;\delta)  \mbox{   in } \mathbb{T}^d ,} \\[2mm]
\displaystyle{  D (\delta \hat{\eta}) \cdot \phi_{m,
\hat{\eta}}(y;\delta) } &=&  0  \mbox{  in }  \mathbb{T}^d , \\  
\displaystyle{ \hat{\eta} \cdot \int\limits_{\mathbb{T}^d}  \phi_{m,\hat{\eta}}(y;\delta)  dy  }&=& 0\\[2mm]
 (\phi_{m,\hat{\eta}}, q_{m,\hat{\eta}})  \mbox{ is } Y &-& \mbox{ periodic.}
\end{array} \right\} \label{eq3.1}
\end{equation}
\textbf{Zeroth order derivatives : }
For $m = 1,  \ldots ,d-1$ and for a fixed direction $\hat{\eta} \in \mathbb{S}^{d -1}$ we have 
$ \lambda_{m,\hat{\eta}} (0) =0$ and a corresponding eigenfunction is such that $q_{m,\hat{\eta}}(y;0) =0, 
q_{0, m,\hat{\eta}}(0) =0$ and  $\phi_{m,\hat{\eta}}(y;0)  $ is a constant unit vector of $\mathbb{R}^d$ orthogonal to 
$\hat{\eta}$. We give a notation for this constant $\phi_{m,\hat{\eta}}(y;0) = \phi_{m,\hat{\eta}}^0$. We recall that $\{ \phi^0_{1,\hat{\eta}},\ldots, \phi^0_{d-1,\hat{\eta}},\hat{\eta}\}$ is such that they form an orthonormal basis for $\mathbb{R}^d$.
\paragraph{First order derivatives : }
Let us differentiate \eqref{eq3.1} once with respect to $\delta$ to obtain (prime denotes derivatives with respect to $\delta$) :
\begin{equation} \label{q-prime}\left. \begin{array}{rllllll}
\displaystyle{ -D (\delta \hat{\eta}) \cdot  (\mu(y)   D (\delta \hat{\eta} )\phi^\prime_{m,\hat{\eta}}(y;\delta))  +  D  (\delta \hat{\eta} ) q^\prime_{m,\hat{\eta}}(y;\delta)  +  q^\prime_{0,m,\hat{\eta}} (\delta)  \hat{\eta} -\lambda_{m,\hat{\eta}}(\delta)\phi^\prime_{m,\hat{\eta}}(y;\delta) }& &\\[2mm]
= \displaystyle{ f(\delta) \mbox{   in } \mathbb{T}^d ,}& & \\[2mm]
\displaystyle{  D (\delta \hat{\eta}) \cdot \phi^\prime_{m,
\hat{\eta}}(y;\delta) } =  g(\delta)  \mbox{  in }  \mathbb{T}^d, & & \\[2mm]  
\displaystyle{ \hat{\eta} \cdot \int\limits_{\mathbb{T}^d}  \phi^\prime_{m,\hat{\eta}}(y;\delta)  dy  }= 0, & &\\[2mm]
(\phi^\prime_{m,\hat{\eta}}, q^\prime_{m,\hat{\eta}})  \mbox{ is } Y - \mbox{ periodic}& &
\end{array} \right\} 
\end{equation} 
where,
\begin{equation}
\begin{aligned}
 f(\delta) &= \lambda^\prime_m(\delta)\phi_{m,\hat{\eta}}(y;\delta)- iq_{m,\hat{\eta}}(y;\delta)\hat{\eta} + i\hat{\eta}\cdot \mu(y) D(\delta\hat{\eta})\phi_{m,\hat{\eta}}(y;\delta) + iD(\delta \hat{\eta})\cdot (\mu(y)\phi_{m,\hat{\eta}}(y;\delta)\otimes \hat{\eta}),\\[2mm]                
          g(\delta) &= -i\hat{\eta}\cdot \phi_{m,\hat{\eta}}(y;\delta).
\end{aligned}
\end{equation}
We put  $\delta =0$ in \eqref{q-prime} and by integrating over $\mathbb{T}^d$, we obtain 
\begin{equation*}
 q^\prime_{0,m,\hat{\eta}}(0)\hat{\eta} = \lambda^\prime_{m,\hat{\eta}}(0)\phi^0_{m,\hat{\eta}}.
\end{equation*}
Taking scalar product with $\hat{\eta}$, we simply get
 $\lambda^\prime_{m,\hat{\eta}}(0) =q^\prime_{0,m,\hat{\eta}}(0)=0$ as $\hat{\eta}\perp \phi^0_{m,\hat{\eta}}$.\\
 \\
Using the above information in \eqref{q-prime}, we find that $(\phi^\prime_{m,\hat{\eta}}(y;0),q^\prime_{m,\hat{\eta}}(y;0))$ is a solution of the following cell problem :
\begin{equation} \label{cell-equation}\left. \begin{array}{rllllll}
\displaystyle{ -\nabla \cdot  (\mu(y)   \nabla \phi^\prime_{m,\hat{\eta}}(y;0))  +  \nabla q^\prime_{m,\hat{\eta}}(y;0) }
&=& \displaystyle{ i\nabla\cdot (\mu(y)\phi^0_{m,\hat{\eta}}\otimes \hat{\eta})  \mbox{   in } \mathbb{T}^d ,} \\[2mm]
\displaystyle{  \nabla \cdot \phi^\prime_{m,
\hat{\eta}}(y;0) } &=&  0  \mbox{  in }  \mathbb{T}^d, \\[2mm]  
\displaystyle{ \hat{\eta} \cdot \int\limits_{\mathbb{T}^d}  \phi^\prime_{m,\hat{\eta}}(y;0)  dy  } &=& 0,\\[2mm] 
\displaystyle{ \int\limits_{\mathbb{T}^d}  q^\prime_{m,\hat{\eta}}(y;0)  dy  } &=& 0\\[2mm]
(\phi^\prime_{m,\hat{\eta}}(y;0), q^\prime_{m,\hat{\eta}}(y;0))  \mbox{ is } Y &-& \mbox{ periodic.}
\end{array} \right\} 
\end{equation} 
Comparing this with \eqref{eq8}, it can be seen that that  $\phi^\prime_{m,\hat{\eta}}(y;0)$ is given by (see \cite{ACFO}) :  
\begin{equation}
\phi_{m,\hat{\eta}}^\prime (y; 0)  = i \hat{\eta}_\alpha  \chi^{r}_\alpha (y) (\phi_{m,\hat{\eta}}^0)_r + \zeta_{m,\hat{\eta}}
\end{equation}
where $\zeta_{m,\hat{\eta}}\in\mathbb{C}^d$ is a constant vector (independent of $y$), orthogonal to $\hat{\eta}$. 
In other words, the $y$-dependence of $\phi_{m,\hat{\eta}}^\prime (y;0)$ is completely determined by the cell test function 
$\chi^r_\alpha(y)$, solution of problem \eqref{eq8}. \\
\\
In a similar manner, the  derivative of the eigenpressure $q_{m,\hat{\eta}}(y;0)$ is given by (see \cite{ACFO}):  
\begin{equation}\label{q-prime-0}
q_{m,\hat{\eta}}^\prime (y; 0) = i \hat{\eta}_\alpha \pi^{r}_\alpha (y) (\phi_{m,\hat{\eta}}^0)_r  \, ,
\end{equation}
That is, the $y$-dependence of $q_{m,\hat{\eta}}^\prime (y;0)$ is completely determined by the cell test function $\pi^r_\alpha(y)$ , solution of problem \eqref{eq8}. 
\paragraph{Second order derivatives : }

Next we differentiate \eqref{q-prime} with respect to $\delta$ to obtain :
\begin{equation} \left. \begin{array}{rllllll}
\displaystyle{ -D (\delta \hat{\eta}) \cdot  (\mu(y)   D (\delta \hat{\eta} )\phi^{\prime\prime}_{m,\hat{\eta}}(y;\delta))  +  D  (\delta \hat{\eta} ) q^{\prime\prime}_{m,\hat{\eta}}(y;\delta)  +  q_{0,m,\hat{\eta}}^{\prime\prime}(\delta)  \hat{\eta} -\lambda_{m,\hat{\eta}}(\delta)\phi^{\prime\prime}_{m,\hat{\eta}}(y;\delta) }& &\\[2mm]
= \displaystyle{  F(\delta)  \mbox{   in } \mathbb{T}^d ,}& & \\[2mm]
\displaystyle{  D (\delta \hat{\eta}) \cdot \phi^{\prime\prime}_{m,\hat{\eta}}(y;\delta) } =  G(\delta)  \mbox{  in }  \mathbb{T}^d , & & \\[2mm]  
\displaystyle{ \hat{\eta} \cdot \int\limits_{\mathbb{T}^d}   \phi^{\prime\prime}_{m,\hat{\eta}}(y;\delta)  dy  } = 0, & &  \\[2mm]
(\phi^{\prime\prime}_{m,\hat{\eta}}, q^{\prime\prime}_{m,\hat{\eta}})  \mbox{ is } Y - \mbox{ periodic}& &
\end{array} \right\} \label{q-double-prime} 
\end{equation}
where 
\begin{equation}\begin{aligned}
 F(\delta)= -2\mu(y)\phi_{m,\hat{\eta}}(y;\delta) &+ 2i\hat{\eta} \cdot\mu(y)D(\delta\hat{\eta})\phi^{\prime}_{m,\hat{\eta}}(y;\delta) + 2i D(\delta \hat{\eta})\cdot (\mu(y)\phi^{\prime}_{m,\hat{\eta}}(y;\delta)\otimes \hat{\eta})\\[2mm]
 &- 2i\hat{\eta}q^\prime_{m,\hat{\eta}}(y;\delta)+\lambda^{\prime\prime}_{m,\hat{\eta}}(\delta)\phi_{m,\hat{\eta}}(y;\delta) + 2\lambda_{m,\hat{\eta}}^\prime(\delta)\phi^\prime_{m,\hat{\eta}}(y;\delta),\\[2mm]
G(\delta) = -2i\hat{\eta}\cdot\phi^\prime_{m,\hat{\eta}}(y;\delta).\\[1mm]
 \end{aligned}\end{equation}
We consider \eqref{q-double-prime} at $\delta=0$ and by integrating over $\mathbb{T}^d$, we get 
\begin{align*}
 q^{\prime\prime}_{0,m,\hat{\eta}}(0)\hat{\eta}_k  &= -\frac{2}{|\mathbb{T}^d|}\int_{\mathbb{T}^d}\mu(y)(\phi_{m,\hat{\eta}}^0)_k\ dy -\frac{2}{|\mathbb{T}^d|}\int_{\mathbb{T}^d} [\hat{\eta}_\beta\mu(y)\nabla_y \chi^l_\beta(y)
 (\phi^0_{m,\hat{\eta}})_l]_{k\alpha}\hat{\eta}_\alpha\ dy\\[2mm]
 &\quad + \lambda^{\prime\prime}_{m,\hat{\eta}}(0)(\phi^0_{m,\hat{\eta}})_k
\end{align*}
or,
\begin{align}\label{second-derivative}
-\frac{1}{2}&\lb q^{\prime\prime}_{0,m,\hat{\eta}}(0)\hat{\eta}_k - \lambda^{\prime\prime}_{m,\hat{\eta}}(0)(\phi^0_{m,\hat{\eta}})_k\rb = \frac{1}{|\mathbb{T}^d|}\int_{\mathbb{T}^d}\mu(y)\left[\delta_{lk}\delta_{\alpha\beta}+ (\nabla\chi^l_\beta)_{k\alpha}  \right]dy\ \hat{\eta}_\alpha\hat{\eta}_\beta(\phi_{m,\hat{\eta}}^0)_l \notag \\[2mm]
 &\qquad\qquad\quad= \frac{1}{|\mathbb{T}^d|}\int_{\mathbb{T}^d}\mu(y)\left[\nabla (y_\beta e_l ) : \nabla( y_\alpha e_k) + \nabla\chi^l_\beta  : \nabla( y_\alpha e_k)\right] dy\ \hat{\eta}_\alpha\hat{\eta}_\beta(\phi_{m,\hat{\eta}}^0)_l \notag\\[2mm]
&\qquad\qquad\quad= (A^{*})^{kl}_{\alpha\beta} \hat{\eta}_\alpha\hat{\eta}_\beta(\phi_{m,\hat{\eta}}^0)_l.\notag \\[2mm]
&\qquad\qquad\quad= [(\phi_{m,\hat{\eta}}^0)^t M(\hat{\eta}, A^{\ast})]_k = [M(\hat{\eta}, A^{\ast})(\phi_{m,\hat{\eta}}^0)]_k
\end{align}
where $M (\hat{\eta}, A^\ast)$ is the symmetric matrix whose entries are given by 
$$ 
M (\hat{\eta}, A^\ast)_{kl} = (A^{*})^{kl}_{\alpha \beta}  \hat{\eta}_\alpha  \hat{\eta}_\beta.
$$
This is nothing but a contraction of the homogenized tensor $A^{*}$.
As a simple consequence of \eqref{second-derivative}, we get
\begin{equation*}
 -\frac{1}{2} q^{\prime\prime}_{0,m,\hat{\eta}}(0) = M(\hat{\eta}, A^{\ast})\phi_{m,\hat{\eta}}^0\cdot \hat{\eta}\quad \mbox{ and }\quad
  \frac{1}{2}\lambda^{\prime\prime}_{m,\hat{\eta}}(0)= M(\hat{\eta}, A^{\ast})\phi_{m,\hat{\eta}}^0\cdot \phi_{m,\hat{\eta}}^0.
\end{equation*}
It is also follows that  $ M(\hat{\eta},A^{\ast})\phi_{m,\hat{\eta}}^0 \perp \phi^0_{m^\prime,\hat{\eta}}$ for all $m\neq m^\prime$.\\
\\
By summarizing the above computations, we have 
\begin{theorem}\label{thm3.1} 
For $m = 1,  \ldots ,d-1$ and for a fixed direction $\hat{\eta} \in \mathbb{S}^{d -1}$ we have 
\begin{enumerate}
\item[(i)]  $ \lambda_{m,\hat{\eta}} (0) =0$ and a corresponding eigenfunction is such that $q_{m,\hat{\eta}}(y;0) =0, 
q_{0, m,\hat{\eta}}(0) =0$ and  $\phi_{m,\hat{\eta}}(y;0) = \phi_{m,\hat{\eta}}^0$ a unit vector orthogonal to $\hat{\eta}$.

\item[(ii)]  $\lambda_{m,\hat{\eta}}^\prime (0)  =0 $ and $ q_{0,m,\hat{\eta}}^\prime(0) = 0$.

\item[(iii)]  The  derivative of the eigenfunction $\phi_{m,\hat{\eta}}(y;
{\delta})$ at $\delta =0$ satisfies:  
$$ 
\phi_{m,\hat{\eta}}^\prime (y; 0)  = i \hat{\eta}_\alpha  \chi^{r}_\alpha (y) (\phi_{m,\hat{\eta}}^0)_r + \zeta_{m,\hat{\eta}}
$$
where $\zeta_{m,\hat{\eta}}\in\mathbb{C}^d$ is a constant vector (independent of $y$), orthogonal to $\hat{\eta}$.

\item[(iv)] The  derivative of the eigenfunction $q_{m,\hat{\eta}}(y;{\delta})$ at $\delta =0$ satisfies:  
$$ 
q_{m,\hat{\eta}}^\prime (y; 0) = i \hat{\eta}_\alpha \pi^{r}_\alpha (y) (\phi_{m,\hat{\eta}}^0)_r  .
$$
\item[(v)]  The second derivative of the eigenvalue $\lambda_{m,\hat{\eta}}({\delta})$  and $q_{0,m,\hat{\eta}}(\delta)$ at $\delta =0$ 
satisfy the relation  
\begin{equation}\label{eq3.10}
\frac{1}{2} \lambda_{m,\hat{\eta}}^{\prime \prime} (0) \phi_{m,\hat{\eta}}^0 =\frac{1}{2} q^{\prime\prime}_{0,m,\hat{\eta}}(0)\hat{\eta} +  M(\hat{\eta},A^{\ast})\phi_{m,\hat{\eta}}^0
\end{equation}
where $M (\hat{\eta}, A^\ast)$ is the symmetric matrix whose entries are given by 
$$ 
M (\hat{\eta}, A^\ast)_{kl} = (A^{*})^{kl}_{\alpha \beta}  \hat{\eta}_\alpha  \hat{\eta}_\beta.
$$
\end{enumerate}
\end{theorem} 

\begin{remark} 
The above matrix $M (\hat{\eta}, A^\ast)$ is precisely that which must be positive definite 
in the Legendre-Hadamard definition of ellipticity.  
A relation analogous to \eqref{eq3.10} is called \lq \lq propagation relation" in \cite{GV} in the study of linearized elasticity system and it shows how the homogenized tensor $A^{*}$ enters into the Bloch wave analysis. The above relation \eqref{eq3.10} generalizes the relation (22) in \cite{ACFO}. 
\end{remark}

\begin{remark}
In the linearized elasticity system, the propagation relation is an eigenvalue relation. 
Here, relation \eqref{eq3.10} can again be seen as an eigenvalue problem, posed in the 
$(d-1)$-dimensional subspace orthogonal to $\hat{\eta}$. More precisely, 
$1/2 \lambda_{m,\hat{\eta}}^{\prime \prime} (0)$ is an eigenvalue and 
$\phi_{m,\hat{\eta}}^0$ (which is orthogonal to $\hat{\eta}$) is an eigenvector 
of the restriction of the matrix $M(\hat{\eta},A^{\ast})$ 
to the subspace $\hat{\eta}^\perp$. In \eqref{eq3.10} $1/2 q^{\prime\prime}_{0,m,\hat{\eta}}(0)$ 
is the Lagrange multiplier corresponding to the constraint that the eigenvalue problem 
is posed in the $(d-1)$-dimensional subspace orthogonal to $\hat{\eta}$.
\end{remark}

\paragraph{ Case of Symmetrized gradient : \\ \\}
We recall the incompressible elasticity system \eqref{sop} with the symmetrized gradient introduced in Section 1.
\begin{equation} \left. \begin{array}{rllllll} 
 -\nabla  \cdot      (\mu^\epsilon E(u_s^\epsilon))  + \nabla p_s^\epsilon &=& f \mbox{  in } \Omega, \\[2mm]
 \nabla \cdot  u_s^\epsilon &=&  0 \mbox{   in } \Omega,  \\[2mm]
 u_s^\epsilon  &=& 0   \mbox{ on }\partial\Omega. \\[2mm]
 \end{array} \right\}  \label{sop1} \end{equation}
where $E (v) =  \frac{1}{2}\lb \nabla v + \nabla^t  v \rb.$\\
\\
We introduce Bloch waves associated to the Stokes operator defined in \eqref{sop1}. \\
Find $\lambda_s = \lambda_s (\eta) \in \mathbb{R},  \phi_s =\phi_s  (\eta)  \in  H^1 (\mathbb{T}^d)^d , \   \phi_s \neq 0$ and  $\pi_s = \pi_s (\eta) \in L^2 (\mathbb{T}^d)$ satisfying 
\begin{equation} \left. \begin{array}{rllllll}
- D (\eta) \cdot (\mu E (\eta) \phi_s) + D (\eta) \pi_s  &=& \lambda_s  (\eta) \phi_s \mbox{ in } \mathbb{R}^d \\[2mm]
 D (\eta) \cdot \phi_s &=& 0 \mbox{ in } \mathbb{R}^d \\[2mm]
 (\phi_s, \pi_s)  \mbox{ is } Y &-& \mbox{ periodic} \\[2mm]
\int\limits_{Y} |\phi_s|^2 dy &=&1.
\end{array} \right\} 	\label{sep}   \end{equation}
\noindent  As usual $D (\eta)  =\nabla+ i\eta$ is the shifted gradient operator and
the shifted strain rate tensor is defined by :
$$ \begin{array}{rlllll}
 2E (\eta)  \psi &=& (\nabla +i \eta) \psi +  (\nabla  +i \eta)^t \psi, \\
\left(2E  (\eta) \psi \right)_{kl}  &=& \left(\frac{\partial \psi_k}{\partial x_l} + i \eta_l  \psi_k 
\right) +   \left(\frac{\partial \psi_l}{\partial x_k} + i \eta_k  \psi_l \right). 
\end{array} $$ 
\noindent
As earlier, we modify the spectral problem \eqref{sep} as follows :
Find  $\lambda_s(\delta) \in  \mathbb{R},   \phi_s(.;\delta) \in  H^1 (\mathbb{T})^d, \ q_s(.;\delta) \in L^2_0 (\mathbb{T}^d)$ and  $q_{0,s}(\delta) \in  \mathbb{C}$ satisfying   
 
\begin{equation} \left. \begin{array}{rllllll}
\displaystyle{ -D (\delta e) \cdot  (\mu(y)   E (\delta e )\phi_s(y;\delta)  )+  D  (\delta e ) q_s(y.;\delta)  +  q_{0,s}(\delta)  e  } &=& \displaystyle{  \lambda_s(\delta)  \phi_s(y;\delta)    \mbox{   in } \mathbb{T}^d } \\[2mm]
\displaystyle{  D (\delta e) \cdot \phi_s(y;\delta) } &=&  0  \mbox{  in }  \mathbb{T}^d \\[2mm]  
\displaystyle{ e \cdot \int\limits_{\mathbb{T}^d}  \phi_s(y;\delta)  dy  }&=& 0, \\[2mm]
(\phi_s, q_s)  \mbox{ is } Y &-& \mbox{ periodic,} \\[2mm]
\displaystyle{ \int\limits_{\mathbb{T}^d} |\phi_s(y;\delta)^2  dy } &=& 1. 
\end{array} \right\} 	\label{sepn}  \end{equation}
As before, we can compute directional derivatives of the solution of \eqref{sepn} and prove a result completely analogous to Theorem \ref{thm3.1}. In particular, we will have the following propagation relation :
For $m = 1,  \ldots d-1$ and for fixed direction 
$\hat{\eta}\in \mathbb{S}^{d -1}$  the second derivative of the eigenvalue $\lambda_{s,m,\hat{\eta}}({\delta})$ at $\delta =0$ 
satisfies the  relation  
\begin{equation}\label{spl}
 \frac{1}{2} \lambda_{s,m,\hat{\eta}}^{\prime \prime} (0) \phi_{s,m,\hat{\eta}}^0 = \frac{1}{2}q_{0,s,m,\hat{\eta}}^{\prime \prime} (0)\hat{\eta} +  M  (\hat{\eta}, A_s^{\ast})\phi_{s,m,\hat{\eta}}^0, 
\end{equation}
where   $M (\hat{\eta}, A_s^\ast) $  is the  matrix  whose   entries are given by 
$$ M  (\hat{\eta}, A_s^\ast)_{jl} =    (A_s^{*})^{jl}_{\alpha \beta}  \hat{\eta}_\alpha  \hat{\eta}_\beta.$$

\section{Recovery of homogenized tensor from Bloch waves }\label{sec4}
\setcounter{equation}{0} 

In the scalar self-adjoint case, it is known that the homogenized matrix is equal 
to one-half the Hessian of the first Bloch eigenvalue at zero momentum \cite{CV}. 
In the general (non-symmetric) scalar case, treated in \cite{GV2}, it was shown that 
only the symmetric part of the homogenized matrix is determined by the Bloch spectrum 
and it is given again by the same one-half of the Hessian of the first Bloch eigenvalue 
(which exists by virtue of the Krein-Rutman theorem). The fact that only the 
symmetric part of the homogenized matrix plays a role is not a big surprise since, 
the homogenized tensor $A^{*}$ being constant, the differential operator
$$
\nabla \cdot A^{*} \nabla = \sum_{k,l=1}^d A^{*}_{kl} \frac{\partial^2}{\partial x_k \partial x_l}
$$
depends only on the symmetric part of $A^{*}$. 

In the case of systems, another phenomenon takes place. For example, the 
linearized elasticity system (in which there are no differential constraints) was treated in \cite{GV} 
where it was recognized that not only Bloch eigenvalues but also Bloch eigenfunctions at zero
momentum are needed to determine the homogenized tensor. More precisely, this connection between 
Bloch eigenvalues and eigenfunctions, on the one hand, and the homogenized tensor, on the other 
hand, was expressed via a relation called {\it propagation relation} in \cite{GV} which 
uniquely determines the homogenized tensor. 

In the case of Stokes system, a new phenomenon arises because of the presence of a differential 
constraint (the incompressibility condition). Even though there is an analogue of the propagation 
relation (see \eqref{eq3.10} above), it does not determine uniquely the homogenized tensor.  
In fact the propagation relation \eqref{eq3.10} is unaltered  if we add a multiple of $I \otimes I$ 
(where $I$ is the $d \times d$ identity matrix) to the homogenized tensor. The homogenized 
Stokes operator clearly remains the same under such an addition since it corresponds to adding 
a gradient of the velocity divergence which vanishes because of the incompressibility constraint.  
The authors in \cite{ACFO} conjectured that the homogenized Stokes tensor is uniquely 
characterized by the propagation relation up to the addition of a term $c (I \otimes I)$ 
(where $c$ is a constant). We prove this assertion in the case of the Stokes system \eqref{sop} 
with a symmetrized gradient. For the other Stokes system \eqref{eq1}, the homogenized tensor 
is not uniquely determined by the propagation relation \eqref{eq3.10}.
In this section, we investigate this non-uniqueness. 
Neverheless, we shall prove that for both Stokes systems the homogenized operators 
\eqref{eq9}, and its equivalent for the symmetric gradient case of \eqref{sop}, 
are uniquely determined. \\

Our concern now is the following question: to what extent do the Bloch spectral elements determine 
the homogenized tensor $A^\ast$ via the propagation relation (\ref{eq3.10})~?  
Since $\lambda^{\prime\prime}_{m,\hat{\eta}}(0), q^{\prime\prime}_{0,m,\hat{\eta}}(0),\phi^0_{m,\hat{\eta}}$ are known from Bloch spectral data, it follows that $M(\hat{\eta},A^{*})\phi^0_{m,\hat{\eta}}$ is uniquely determined via the relation (\ref{eq3.10}). But it may happen that different tensors $A^\ast$ give rise to the same matrix $M(\hat{\eta},A^{*})$. Three main results are proved in this section and they are stated in the following three propositions. 

\begin{proposition}\label{prop 3.1}
Let $A^{*}$ and $B^{*}$ be two fourth order tensors possessing the simple symmetry \eqref{eq9A}. 
They satisfy the same propagation relation \eqref{eq3.10}, 
if and only if
\begin{equation}\label{eq.AB}
  B^{*} - A^{*} = c (I \otimes I) + N
\end{equation}
where $I$ is the $d \times d$ identity matrix and $N$ is a fourth order tensor satisfying, 
on top of the simple symmetry \eqref{eq9A}, the following anti-symmetry property
\begin{equation} \left. \begin{array}{rllllll}
N^{jl}_{\alpha\beta} &=& - N^{jl}_{\beta\alpha} = -N^{lj}_{\alpha\beta} \quad\mbox{whenever, }  (\alpha, \beta) \neq (j,  l) \mbox{ and }  (\beta, \alpha ) \neq (j, l) \\[2mm]
N^{ii}_{ii}&=& 0.
\end{array} \right\} \label{Np}  
\end{equation}
\end{proposition} 

\begin{proof}
First of all, let us note that the addition of $c (I \otimes I)$ and $N$, 
having properties \eqref{eq9A} and \eqref{Np}, to $A^\ast$ does not alter the propagation 
relation (\ref{eq3.10}). Indeed, we have, 
$$
\begin{array}{lllll}
M (\hat{\eta}, A^\ast + c (I \otimes I) + N )_{jl} &=& {(A^{*})}_{\alpha \beta}^{jl} \hat{\eta}_\alpha \hat{\eta}_\beta 
+ c \delta_{\alpha j} \delta_{\beta l} \hat{\eta}_\alpha \hat{\eta}_\beta 
+ N^{jl}_{\alpha\beta}\hat{\eta}_\alpha  \hat{\eta}_\beta\\[2mm] 
 &=& M  (\hat{\eta},  A^\ast)_{jl}   + c {\hat{\eta}}_j {\hat{\eta}}_l  \, .
\end{array}
$$
Since $\phi_{m,\hat{\eta}}^0$ is orthogonal to $\hat{\eta}$, we deduce 
$$
M (\hat{\eta}, A^\ast + c (I \otimes I) + N) \phi_{m,\hat{\eta}}^0 = M (\hat{\eta}, A^\ast) \phi_{m,\hat{\eta}}^0 . 
$$ 
Conversely, let us assume that there are two fourth-order tensors $A^\ast$ and $B^\ast$, 
possessing the simple symmetry (\ref{eq9A}) and such that 
$M(\hat{\eta},A^{*}) \phi_{m,\hat{\eta}}^0 = M(\hat{\eta},B^{*}) \phi_{m,\hat{\eta}}^0$, $m=1,...,d-1$, 
for all $\hat{\eta}\in\mathbb{S}^{d-1}$. We must then deduce \eqref{eq.AB}.   
For convenience,  the proof is divided into five steps.\\

\noindent {\bf Step 1.} 
First of all, we check that the matrix $M(\hat{\eta}, A^\ast)$ is symmetric. 
By interchanging the dummy indices $\alpha$ and $\beta$ and using the simple symmetry (\ref{eq9A}) 
of the homogenized coefficients, ${(A^{*})}^{jl}_{\alpha\beta} = {(A^{*})}^{lj}_{\beta\alpha}$, we get
\begin{equation}\label{symmetric}
M(\hat{\eta}, A^\ast)_{jl} = (A^{*})^{jl}_{\alpha \beta}  \hat{\eta}_\alpha  \hat{\eta}_\beta = {(A^{*})}^{jl}_{\beta\alpha}\hat{\eta}_\beta\hat{\eta}_\alpha={(A^{*})}^{lj}_{\alpha\beta}\hat{\eta}_\alpha\hat{\eta}_\beta = M(\hat{\eta}, A^\ast)_{lj} 
\end{equation}
which shows the required symmetry.\\

\noindent {\bf Step 2.} 
For $\widetilde{N}= B^\ast -A^\ast$ define $M (\hat{\eta}) = M (\hat{\eta},\widetilde{N}) = M (\hat{\eta} , B^\ast) - M (\hat{\eta} , A^\ast)$.   
Since $A^\ast$ and $B^\ast$ satisfy (\ref{eq3.10}), it follows that $M(\hat{\eta}) \phi^0_{m,\hat{\eta}} = 0$ for 
$m=1,...,d-1$. Since the family $\phi^0_{m,\hat{\eta}}$ is a basis of the orthogonal space to $\hat{\eta}$, it 
implies that $M (\hat{\eta}) = c(\hat{\eta}) \hat{\eta} \otimes \hat{\eta}$ for some scalar $c(\hat{\eta})$. 
Since $M (\hat{\eta})$ depends quadratically on $\hat{\eta}$, it must be that $c(\hat{\eta})$ is
independent of $\hat{\eta}$. Thus, for $c\in\mathbb{R}$, we have $M (\hat{\eta}) = c \, \hat{\eta} \otimes \hat{\eta}$, 
that is, for any $\hat{\eta}\in\mathbb{S}^{d-1}$,
\begin{equation}
\widetilde{N}_{\alpha \beta}^{jl} \hat{\eta}_{\alpha} \hat{\eta}_{\beta}  = c \hat{\eta}_j \hat{\eta}_l \quad 
1 \leq j, l \leq d . 
\label{eq3.3} 
\end{equation}

\noindent {\bf Step 3.}  Under condition \eqref{eq3.3}, we verify that 
\begin{eqnarray}
\widetilde{N}_{ii}^{ii}  &=& c  \  \ \forall \  \   i .\label{eq3.4}  \\
\mbox{and }\quad \widetilde{N}_ {ik}^{jl} + \widetilde{N}_{ki}^{jl} &=& 0 \mbox{ if }  (i, k) \neq (j,  l) \mbox{ and }  (k, i ) \neq (j, l). \label{eq3.5} 
\end{eqnarray}
\noindent For this purpose, let us take $\hat{\eta} =  e_i$ in (\ref{eq3.3}). We obtain
 $\widetilde{N}_{ii}^{jl}= c \delta_{ij} \delta_{il} $ and so
\begin{eqnarray}
 \widetilde{N}_{ii}^{ii} &=& c \\
 \mbox{and}\quad  \widetilde{N}_{ii}^{jl}  &=& 0 \mbox{ if } i \neq j  \mbox{ or } i \neq  l. \label{eq3.6} 
 \end{eqnarray}
In particular, \eqref{eq3.4} is proved. 
Next, choosing $\hat{\eta} =e_{i}  +e_k$ in  (\ref{eq3.3}), we get   
\begin{equation}
\widetilde{N}_{ii}^{jl}  +\widetilde{N}_{kk}^{jl}  + \widetilde{N}_{ik}^{jl} +\widetilde{N}_{ki}^{jl}  = c (\delta_{ji}  + \delta_{jk})  (\delta_{li} + 
\delta_{lk}). \label{eq3.7} 
\end{equation}
To check \eqref{eq3.5}, there are several cases to consider. 
\begin{enumerate}
\item[(i)]  $(i\neq  j$  and $ k\neq j)$. In this case, (\ref{eq3.5}) is a direct consequence   of 
(\ref{eq3.6}) and (\ref{eq3.7}). 
\item[(ii)] Similarly, for $(k\neq  l$  and $ i\neq l)$ (\ref{eq3.5}) is a direct consequence   of 
(\ref{eq3.6}) and (\ref{eq3.7}).
\item[(iii)]  $(i \neq j, \  k=j ).$ In this case, 
\begin{equation}
\widetilde{N}_{jj}^{jl}  + \widetilde{N}_{ij}^{jl} + \widetilde{N}_{ji}^{jl}  = c (\delta_{li}  + \delta_{lj}).  
\end{equation}
Now together with $i \neq l$ we have
\begin{equation}
 \widetilde{N}_{jj}^{jl}  + \widetilde{N}_{ij}^{jl} + \widetilde{N}_{ji}^{jl}   = c\delta_{lj}.  
\end{equation}
Then both $j=l$ or $j\neq l$ cases lead to verify \eqref{eq3.4} and \eqref{eq3.5} respectively. 

\item[(iv)] Similarly, for $(k\neq  l$  and $ i = l)$ 
\begin{equation}
\widetilde{N}_{ii}^{ji}  + \widetilde{N}_{ik}^{ji} + \widetilde{N}_{ki}^{ji}  = c (\delta_{ji}  + \delta_{jk}). 
\end{equation}
Together with $k \neq j$ we have
\begin{equation}
 \widetilde{N}_{ii}^{ji}  + \widetilde{N}_{ik}^{ji} + \widetilde{N}_{ki}^{ji}   = c\delta_{ji}.  
\end{equation}
Then both $i=j$ or $i\neq j$ cases lead to verify \eqref{eq3.4} and \eqref{eq3.5} respectively. 
\end{enumerate}
\noindent {\bf Step 4.} 
Now we consider the two remaining cases not covered in \eqref{eq3.5}.
\begin{enumerate}
\item[(i)] $(i, k) = (j, l)$.  Then from \eqref{eq3.7} we have
$$ \widetilde{N}_{ii}^{ik}  +\widetilde{N}_{kk}^{ik}  + \widetilde{N}_{ik}^{ik} + \widetilde{N}_{ki}^{ik}  = c (1  + \delta_{ik})^2.  $$
For $i\neq k$ it gives using \eqref{eq3.6}
\begin{equation} 
\widetilde{N}_{ik}^{ik} + \widetilde{N}_{ki}^{ik}  = c. 
\label{r1}\end{equation}
\item[(ii)] Similarly, for $(k,i) = (j,l)$, together with $i \neq k$ we have
\begin{equation} 
\widetilde{N}_{ik}^{ki} + \widetilde{N}_{ki}^{ki} = c 
\label{r2}\end{equation}
\end{enumerate}
\noindent {\bf Step 5.}  Let us set $N  = \widetilde{N} - c  (I \otimes   I).$ Thanks  to  the properties  
(\ref{eq3.4}) and  (\ref{eq3.5}), we can easily check that $N$ is an 
anti-symmetric tensor  in the sense   that it satisfies  
\begin{equation}\label{nedit}
N_{ik}^{jl} =  -N_{ki}^{jl} = -N_{ik}^{lj}. \quad\mbox{whenever, }  (i, k) \neq (j,  l) \mbox{ and }  (k, i ) \neq (j, l)
\end{equation}
From its very definition  $N$ also possesses the symmetry 
$ N_{ik}^{jl} = N_{ki}^{lj}.$ 
Thus $N$ has all the properties listed in \eqref{Np}.
\end{proof}

\noindent
Next we extend Proposition \ref{prop 3.1} to the Stokes system \eqref{sop}, 
featuring a symmetric gradient tensor. In this case the propagation relation 
\eqref{eq3.10} is replaced by \eqref{spl} and the homogenized tensor is 
denoted by $A^{*}_s$.

\begin{proposition}
The propagation relation \eqref{spl} characterizes uniquely the tensor $A_s^\ast$, 
up to the addition of a constant multiple of $I \otimes I$. In other words, 
$A^{*}_s$ and $B^{*}_s$ satisfy the same propagation relation \eqref{spl} 
if and only if, for some $c\in \mathbb{R}$, 
\begin{equation}\label{sn} 
  B_s^{*} - A_s^{*} = c (I \otimes I).
\end{equation}
\end{proposition}

\begin{proof}
The proof continues from the \noindent {\bf Step 5} of the previous proof of Proposition \ref{prop 3.1}.
We defined $N = \widetilde{N} - c(I\otimes I)$ satisfying \eqref{nedit} i.e.
$$N_{ik}^{jl} =  -N_{ki}^{jl} = -N_{ik}^{lj}. \quad\mbox{whenever, }  (i, k) \neq (j,l) \mbox{ and }  (k, i ) \neq (j, l)$$
Now as $\widetilde{N}= B^{*}_s - A^{*}_s$ possess with the symmetry of coefficients of linear elasticity, so we have  
\begin{equation}\label{nedit2} N_{ik}^{jl} =  N_{jk}^{il} =  N_{ki}^{lj} = N_{il}^{jk} \quad\mbox{for all }i,j,k,l. \end{equation}
This symmetry combined  with the  anti-symmetry established in the previous step implies  that $N =0.$ 
Note that antisymmetry property holds precisely for the interchange of those pairs of indices for which symmetry property does not hold.\\
This can  be seen as follows: whenever $ (i, k) \neq (j,  l) \mbox{ and }  (k, i ) \neq (j, l)$
\begin{align}
&N_{ik}^{jl} =  -N_{ik}^{lj} = -N_{lk}^{ij} = N_{lk}^{ji} =N_{kl}^{ij} =  N_{il}^{kj} = -N_{il}^{jk} = - N_{ik}^{jl} \\
\mbox{Thus }\quad &N_{ik}^{jl} =0.\label{nedit3}
\end{align}
Similarly, whenever $ (i, k) = (j,  l) \mbox{ or }  (k, i ) = (j, l)$ together with $i\neq k$;
from \eqref{r1}, \eqref{r2} we have 
$$ 
\widetilde{N}_{ik}^{ik} + \widetilde{N}^{ik}_{ki} \ = \ c\ = \ \widetilde{N}_{ik}^{ki} + \widetilde{N}_{ki}^{ki}. 
$$
Then using \eqref{nedit2} and \eqref{eq3.6} we clearly have 
\begin{equation}\label{nedit4} 
N_{ik}^{ik} = \ 0\ = \ N_{ki}^{ki}. 
\end{equation}
Therefore \eqref{nedit3}, \eqref{nedit4} imply that $N =0$ or, $\widetilde{N} =  c (I\otimes  I)$ and hence  $B_s^\ast  -A_s^\ast  = c (I \otimes I). $
\end{proof}

\begin{remark}
The conclusion of the above proposition was conjectured in \cite{ACFO} and it is proved here to be true whenever we are working with the system \eqref{sop} with symmetrized gradient. 
However, it is not true with the full gradient Stokes system \eqref{eq1} as shown by Proposition \ref{prop 3.1}. However, in both of these cases the propagation relation fixes the homogenized operator \eqref{eq9} uniquely, as is stated in the following proposition.
\end{remark}

\begin{proposition}
If \eqref{eq.AB} is satisfied, then $A^{*}$ and $B^{*}$ give rise to the same homogenized operator \eqref{eq9}.
\end{proposition}

\begin{proof}
We have to check that $A^{*}$ and $B^{*}$ define the same Stokes differential operator 
for divergence-free vector fields. Indeed the Fourier symbol of the operator 
$$
u = (u_k)_{1\leq k\leq d} \ \to \ \left( 
-\frac{\partial}{\partial x_\beta} \left((A^{*}-B^{*})^{kl}_{\alpha \beta} \frac{\partial u_k}{\partial x_\alpha} \right) \right)_{1\leq l \leq d}
$$
is $(A^{*}-B^{*})^{kl}_{\alpha \beta}\xi_\alpha\xi_\beta$ which, by virtue of \eqref{eq3.3}, is equal to $c\xi_k\xi_l$ 
which is precisely the symbol of the operator $u \to -c \nabla (\nabla\cdot u)$ which vanishes 
on the space of divergence free functions. 
\end{proof}

\section{Homogenization result}\label{sec5}
\setcounter{equation}{0} 

This section is devoted to a proof of Theorem \ref{thm1.1}, 
our main homogenization result stated in the first section. 
It is based on the tools  that   we have introduced so far. 
A similar  proof is given for the linear elasticity problem in \cite{GV2}. 
However, the presence of a pressure and a differential constraint 
in the Stokes system seriously complexifies the analysis 
and has a non-trivial effect in the homogenization process.  
Besides, we also bring some simplifications to the proof given in \cite{GV2}.\\

We consider  a sequence of  solutions  $(u^{\epsilon},p^{\epsilon})\in (H^1_{0}(\Omega))^d\times L^2_0 (\Omega) $ solving  the Stokes system \eqref{eq1}.   It is classical to derive the following bound:  
\begin{equation}\label{estimate-u-p}
||u^\epsilon||_{(H^1_{0}(\Omega))^d} + ||p^\epsilon||_{L^2(\Omega)} \leq C ||f||_{(L^2(\Omega))^d}, 
\end{equation} 
where $C$ is  independent of $\epsilon.$ 
Then there exist $(u,p)\in (H^1_{0}(\Omega))^d\times L^2_0 (\Omega) \ $ and a subsequence ($u^{\epsilon}$,  $p^{\epsilon}$) converging weakly 
to  $(u, p)$ in $(H^1_{0}(\Omega))^d  \times L^2_0(\Omega)$. Our aim is to show that $(u,p)$
satisfies the homogenized Stokes system \eqref{eq9}. Due to the uniqueness of solutions for the system \eqref{eq9}, it follows that the entire sequence  $(u^{\epsilon},p^{\epsilon})$ converges  to $ (u,p) $ weakly in $(H^1_0 (\Omega))^d \times L^2_0 (\Omega).$ \\

There are several steps in the proof. 
First, we localize the Stokes system \eqref{eq1} by applying a cut-off function technique 
to the velocity $u$ in order to get the equation \eqref{local} in the whole $\mathbb{R}^d$. 
Next, by taking  the Bloch transformation $B_{m,\hat{\eta}}^{\epsilon}$ $(1\leq m \leq d-1 )$ 
of the equation \eqref{local} and passing to the limit,  we arrive at  the 
homogenized equation in the Fourier space. 
Finally, we take the inverse Fourier transform to go back to the  physical space which gives our desired result. \\

\noindent{\bf Notation:} in the sequel L.H.S. stands for left hand side, 
and R.H.S. for right hand side. 

\paragraph{Step 1. Localization of the velocity $u$ : }
Let $v \in \mathcal{D}(\Omega)$ be arbitrary.
Then $vu^{\epsilon}$ and $p^{\epsilon}$ satisfy (for $l=1, \ldots , d)$ 
\begin{equation}\label{local} 
-\frac{\partial}{\partial x_\alpha}(\mu^{\epsilon}\frac{\partial}{\partial x_\alpha})(vu^{\epsilon}_l) +  \frac{\partial p^{\epsilon}}{\partial x_l}v = vf_l + g^{\epsilon}_l + h^{\epsilon}_l\quad\mbox{in }\mathbb{R}^d,   
\end{equation}
where,
\begin{equation}
g_l^{\epsilon} = -2\mu^{\epsilon}\frac{\partial u^{\epsilon}_l}{\partial x_\alpha}\frac{\partial v}{\partial x_\alpha} - \mu^{\epsilon}\frac{\partial^2v}{\partial x_\alpha \partial x_\alpha}u^{\epsilon}_l
\quad\mbox{and}\quad h_l^{\epsilon}  = -\frac{\partial \mu^{\epsilon}}{\partial x_\alpha}\frac{\partial v}{\partial x_\alpha}u^{\epsilon}_l.
\end{equation}
Note that, $g_l^{\epsilon}$ and $h_l^{\epsilon}$
correspond to terms containing zero and first order derivatives of $\mu^{\epsilon}$ respectively. In the sequel, we extend $u^\epsilon$ and $p^\epsilon$ by  zero outside $\Omega$ and such extensions are denoted by the same letters.

\paragraph{Step 2. Limit of $B_{m,\hat{\eta}}^{\epsilon}$ applied to the L.H.S. of \eqref{local} :}
We consider the following $\epsilon$-scaled spectral problem of \eqref{eq3.1} as follows : Let
$\hat{\eta}= \frac{\xi}{|\xi|}  \in \mathbb{S}^{d-1}$, $\delta = \epsilon(\xi \cdot\hat{\eta})$;   
\begin{equation*}\begin{aligned}
\phi_{m,\hat{\eta}}^{\epsilon}(x;\delta) &= \phi_{m,\hat{\eta}}(\frac{x}{\epsilon};\epsilon(\xi\cdot\hat{\eta})), \ \mbox{and}\ \lambda_{m,\hat{\eta}}^{\epsilon}(\delta) = \epsilon^{-2}\lambda_{m, \hat{\eta}}(\epsilon (\xi\cdot\hat{\eta}))\\[1mm]
 q_{m,\hat{\eta}}^{\epsilon}(x;\delta) &= \epsilon^{-1}q_{m,\hat{\eta}}(\frac{x}{\epsilon};\epsilon(\xi\cdot\hat{\eta})), \ \mbox{and}\ q^\epsilon_{0,m,\hat{\eta}}(\delta) = \epsilon^{-2}q_{0,m,\hat{\eta}}(\epsilon(\xi\cdot\hat{\eta})).\\[1mm]
\end{aligned}\end{equation*}
They satisfy the following system because of \eqref{eq3.1} :
\begin{equation} \left. \begin{array}{rllllll}
\displaystyle{ -D (\delta \hat{\eta}) \cdot  (\mu^\epsilon(x)   D (\delta \hat{\eta} )\phi^\epsilon_{m,\hat{\eta}}(x;\delta))  +  D  (\delta \hat{\eta} ) q^\epsilon_{m,\hat{\eta}}(x;\delta)  +  q^\epsilon_{0,m,\hat{\eta}}(\delta)\hat{\eta}} &=& \displaystyle{  \lambda^\epsilon_{m,\hat{\eta}}(\delta)  \phi^\epsilon_{m,\hat{\eta}}(x;\delta)  \mbox{   in } \mathbb{R}^d ,} \\[2mm]
\displaystyle{  D (\delta \hat{\eta}) \cdot \phi^\epsilon_{m,
\hat{\eta}}(x;\delta) } &=&  0  \mbox{  in }  \mathbb{R}^d , \\  
\displaystyle{ \hat{\eta} \cdot \int\limits_{\mathbb{R}^d}  \phi^\epsilon_{m,\hat{\eta}}(x;\delta)  dx  }&=& 0, \\
 (\phi^\epsilon_{m,\hat{\eta}}, q^\epsilon_{m,\hat{\eta}}
 )  \mbox{ is } \epsilon Y &-& \mbox{ periodic,} \\[2mm]
\displaystyle{ \int\limits_{\epsilon\mathbb{T}^d} |\phi^\epsilon_{m,\hat{\eta}}(x;\delta)|^2  dx } &=& 1.
\end{array} \right\} \label{epsilon-spectral}
\end{equation}
Let us first consider the L.H.S. of \eqref{local}.
For $g\in H^1(\mathbb{R}^d)^d$ with compact support  in $\Omega$,   using  the definition Bloch transformation \eqref{bloch-transformation2} and spectral equation \eqref{epsilon-spectral},  we obtain for $m= 1, \ldots d-1,$
\begin{equation*}\begin{aligned}
B_{m,\hat{\eta}}^{\epsilon}\lb-\frac{\partial}{\partial x_\alpha}(\mu^{\epsilon}\frac{\partial}{\partial x_\alpha})g\rb(\xi) &= \left\langle e^{i x\cdot\xi}\phi_{m,\hat{\eta}}^{\epsilon}(.;\delta),    -\frac{\partial}{\partial x_\alpha}(\mu^{\epsilon}\frac{\partial}{\partial x_\alpha})g\right\rangle\\[3mm]
&= \left\langle g, -\frac{\partial}{\partial x_\alpha}(\mu^{\epsilon}\frac{\partial}{\partial x_\alpha})(e^{i x\cdot\xi}\phi_{m,\hat{\eta}}^{\epsilon}(.;\delta))\right\rangle\\[3mm]
&= \left\langle g, \lambda_{m,\hat{\eta}}^{\epsilon} (\delta) e^{i x\cdot\xi}\phi_{m,\hat{\eta}}^{\epsilon}(.;\xi)-\nabla( q_{m,\hat{\eta}}^{\epsilon}(.;\delta)e^{i x\cdot\xi})-q_{0,m,\hat{\eta}}^\epsilon(\delta)\hat{\eta} e^{ix\cdot\xi}\right\rangle\\[3mm]
&= \lambda_{m,\hat{\eta}}^{\epsilon} (\delta) B_{m,\hat{\eta}}^{\epsilon}g(\xi)- \left\langle g, \nabla(q_{m,\hat{\eta}}^{\epsilon}(.;\delta)e^{i x\cdot\xi})\right\rangle - \left\langle g, q^\epsilon_{0,m,\hat{\eta}}(\delta)\hat{\eta}e^{ix\cdot\xi}\right\rangle.\\[1mm]
\end{aligned}\end{equation*} 
In the previous equation the duality bracket is between $H^1_{comp}(\mathbb{R}^d)^d$ and $H^{-1}_{loc}(\mathbb{R}^d)^d$. 

Therefore, $B_{m,\hat{\eta}}^{\epsilon}$ applied to the L.H.S. of \eqref{local} ($1\leq m \leq d-1$) is equal to
\begin{equation}\label{bt-lhs}
\lambda_{m,\hat{\eta}}^{\epsilon} (\delta) B_{m,\hat{\eta}}^{\epsilon}(vu^{\epsilon})(\xi)- \left\langle vu^\epsilon, \nabla(q_{m,\hat{\eta}}^{\epsilon}(.;\delta)e^{i x\cdot\xi})\right\rangle - \left\langle vu^\epsilon, q^\epsilon_{0,m,\hat{\eta}}(\delta)\hat{\eta}e^{ix\cdot\xi}\right\rangle + B_{m,\hat{\eta}}^{\epsilon}(v\nabla p^{\epsilon}) (\xi) . 
\end{equation}
Below, we treat each term of  \eqref{bt-lhs} one by one.\\

\noindent{\bf 1st term of \eqref{bt-lhs} : } 
By using the Taylor expansion 
\begin{equation}
 \lambda^\epsilon_{m,\hat{\eta}}(\delta) = \epsilon^{-2}\lambda_{m,\hat{\eta}}(\epsilon(\xi\cdot\hat{\eta})) = \frac{1}{2}\lambda^{\prime\prime}_{m,\hat{\eta}}(0)(\xi\cdot\hat{\eta})^{2} + \mathcal{O}(\epsilon(\xi\cdot\hat{\eta}))
\end{equation}
and then using Theorem \ref{thm2.2}, we get
\begin{equation}
\chi_{\epsilon^{-1}\mathbb{T}^d}(\xi)\lambda_{m,\hat{\eta}}^{\epsilon}(\delta)B_{m,\hat{\eta}}^{\epsilon}(vu^{\epsilon})(\xi) \rightarrow \frac{1}{2}\lambda_{m,\hat{\eta}}^{\prime\prime}(0)(\xi \cdot \hat{\eta})^2\phi_{m,\hat{\eta}}^{0}\cdot\widehat{(vu)}(\xi)\quad\mbox{in }L^2_{loc}(\mathbb{R}_{\xi}^d)\mbox{ strongly,}
\end{equation}
where we recall that    $\phi_{m,\hat{\eta}}^{0}$ is a constant unit vector of $\mathbb{R}^d$ orthogonal to $\hat{\eta}$.
Note that $\lambda_{m,\hat{\eta}}^{\prime\prime}(0)$ is linked to $A^{*}$ via the propagation relation \eqref{eq3.10}.  Using this  relation, the above limit can be written  as 
\begin{align}
&(\xi \cdot \hat{\eta})^2\lb\frac{1}{2}q^{\prime\prime}_{0,m,\hat{\eta}}(0)\hat{\eta} +  M (\hat{\eta}, A^\ast ) \phi_{m,\hat{\eta}}^{0}\rb\cdot\widehat{(vu)}(\xi)\notag\\[2mm]
&= (\xi \cdot \hat{\eta})^2\frac{1}{2}q^{\prime\prime}_{0,m,\hat{\eta}}(0) \hat{\eta}_k\widehat{(vu_k)} + (\xi \cdot \hat{\eta})^2 (A^{\ast})^{ kl}_{\alpha \beta} 
\hat{\eta}_{\alpha} \hat{\eta}_\beta ( \phi^0_{m, \hat{\eta}})_l  (\widehat{v u_k})(\xi).
\end{align}
\noindent{\bf 2nd term of \eqref{bt-lhs} : } 
\begin{align}\label{q}
-\left\langle vu^{\epsilon},\nabla(q_{m,\hat{\eta}}^{\epsilon}e^{i x\cdot\xi})\right\rangle
&= \left\langle \nabla \cdot (vu^{\epsilon}),e^{i x\cdot\xi}q_{m,\hat{\eta}}^{\epsilon}\right\rangle \notag\\[2mm]
&= \left\langle u^{\epsilon} \cdot \nabla v ,e^{i x\cdot\xi}q_{m,\hat{\eta}}^{\epsilon}\right\rangle \quad\mbox{ (as }\nabla\cdot u^{\epsilon} = 0). 
\end{align}
Using  the Taylor expansion of $q^\epsilon_{m,\hat{\eta}}(.;\delta)$ :
\begin{align}
q_{m,\hat{\eta}}^{\epsilon}(x;\delta) &= \epsilon^{-1}q_{m,\hat{\eta}}(\frac{x}{\epsilon};\epsilon(\xi\cdot\hat{\eta}))\notag\\[2mm]
                       &=\epsilon^{-1} q_{m,\hat{\eta}}(\frac{x}{\epsilon};0) + (\xi\cdot\hat{\eta})q^{\prime}_{m,\hat{\eta}}(\frac{x}{\epsilon};0) +\mathcal{O}(\epsilon(\xi\cdot\hat{\eta})^2),  
\end{align}
(prime denotes the derivative with respect to the second variable),
with the properties that (cf. Theorem \ref{thm2.1}) 
\begin{equation}\begin{aligned}
& q_{m,\hat{\eta}}(\frac{x}{\epsilon};0) =0 \ \mbox{ and }\\[1mm] 
& q_{m,\hat{\eta}}^{\prime}(\frac{x}{\epsilon};0) \rightharpoonup M_{\mathbb{T}^d}(q_{m,\hat{\eta}}^{\prime}(y;0)) =0\  \mbox{ weakly in }L^2(\mathbb{R}^d); \ \mbox{ (as }q^\prime_{m,\hat{\eta}}(y;0)\in L^2_0(\mathbb{T}^d))   
\end{aligned}\end{equation}  
where,  $M_{\mathbb{T}^d}(f) = \frac{1}{|\mathbb{T}^d|}\int_{\mathbb{T}^d} f(y)dy$.\\
\\
Then by using $u^\epsilon\rightarrow u $ strongly in $L^2(\Omega)^d$ from \eqref{q} we get
\begin{equation}
-\left\langle vu^{\epsilon},\nabla(q_{m,\hat{\eta}}^{\epsilon}e^{i x\cdot\xi})\right\rangle
\rightarrow \left\langle u\cdot\nabla v, e^{ix\cdot\xi}M_{\mathbb{T}^d}(q^\prime_{m,\hat{\eta}})\right\rangle =0\  \mbox{ in  } L^2_{loc} (\mathbb{R}^d_\xi) \mbox{ strongly.}\\[1mm]
 \end{equation}
It is also used that, the error term $\mathcal{O}(\epsilon(\xi\cdot\hat{\eta})^2)$  in the above Taylor expansion tends to $0$ in the space $L^2_{loc} 
 (\mathbb{R}^d_{\xi}; L^2_{loc} (\mathbb{R}^d))$. Thus the oscillating eigen-pressure $q^\epsilon_{m,\hat{\eta}}$ does not contribute to the homogenized system. \\

\noindent{\bf 3rd term of \eqref{bt-lhs} : }  
We use the Taylor expression of $q_{0,m,\hat{\eta}}^\epsilon(\xi)$ with the property $q_{0,m,\hat{\eta}}(0)=q^{\prime}_{0,m,\hat{\eta}}(0)=0$ (cf. Theorem \ref{thm3.1}) to have
\begin{equation}
 q_{0,m,\hat{\eta}}^\epsilon(\delta) = \epsilon^{-2}q_{0,m,\hat{\eta}}(\epsilon(\xi\cdot\hat{\eta})) = \frac{1}{2}q^{\prime\prime}_{0,m,\hat{\eta}}(0)+\mathcal{O}(\epsilon(\xi\cdot\hat{\eta})^2).
\end{equation}
So,
\begin{align}
 - \left\langle vu^\epsilon, q^\epsilon_{0,m,\hat{\eta}}(\delta)e^{ix\cdot\xi}\hat{\eta}\right\rangle \rightarrow & -\langle vu, \frac{1}{2}q^{\prime\prime}_{0,m,\hat{\eta}}(0)(\xi\cdot\hat{\eta})^2\hat{\eta}e^{ix\cdot\xi}\rangle 
 \ \mbox{ in  } L^2_{loc} (\mathbb{R}^d_\xi) \mbox{ strongly.}\notag\\[1mm]
  &= -  \frac{1}{2}q^{\prime\prime}_{0,m,\hat{\eta}}(0)(\xi\cdot\hat{\eta})^2\widehat{(vu)}\cdot\hat{\eta}.
\end{align} 
 
\noindent{\bf 4th term of \eqref{bt-lhs} : }  Finally, we consider the remaining fourth term in \eqref{bt-lhs}, and doing integration by parts we get 
\begin{align}\label{pr1}
 B^\epsilon_{m, \hat{\eta}}(v \nabla p^\epsilon)(\xi) &= \left\langle v\nabla p^{\epsilon},e^{i x\cdot\xi}\phi_{m,\hat{\eta}}^{\epsilon}\right\rangle\notag\\[2mm]
 &=-\left\langle  p^{\epsilon}, \nabla v \cdot e^{i x\cdot\xi}\phi_{m,\hat{\eta}}^{\epsilon}\right\rangle\quad(\mbox{as }\nabla\cdot (e^{i x\cdot\xi}\phi_{m,\hat{\eta}}^\epsilon) = 0).
\end{align}
We use the Taylor expansion
\begin{align}
\phi_{m,\hat{\eta}}^{\epsilon}(x;\xi) &= \phi_{m,\hat{\eta}}(\frac{x}{\epsilon};0) + \epsilon(\xi\cdot\hat{\eta})\phi^{\prime}_m(\frac{x}{\epsilon};0) + \mathcal{O}((\epsilon(\xi\cdot\hat{\eta}))^2)\notag\\
& = \phi_{m,\hat{\eta}}^{0}  + \epsilon(\xi\cdot\hat{\eta})\phi^{\prime}_{m,\hat{\eta}}(\frac{x}{\epsilon};0) + \mathcal{O}((\epsilon(\xi\cdot\hat{\eta}))^2)\rightarrow  \phi_{m,\hat{\eta}}^{0}  \quad\mbox{in }L^2_{loc}(\mathbb{R}^d_{\xi}, (L^2(\Omega))^d)  \mbox{ strongly.}
\end{align}
And from \eqref{estimate-u-p} as $||p^\epsilon||_{L^2(\Omega)}$ is uniformly bounded, so up to a subsequence we have
\begin{equation}
p^{\epsilon} \rightharpoonup p \mbox{ in } L^2(\Omega).
\end{equation}
Thus by passing to the limit in the R.H.S. of \eqref{pr1}, we get
\begin{align} 
-\left\langle p^{\epsilon}, \nabla v\cdot e^{i x\cdot\xi}\phi_{m,\hat{\eta}}^{\epsilon} \right\rangle \rightarrow & -\left\langle p,\nabla v\cdot e^{i x\cdot\xi}\phi_{m,\hat{\eta}}^{0}\notag \right\rangle\\[2mm]
&= \left\langle \nabla p, ve^{i x\cdot\xi}\phi_{m,\hat{\eta}}^{0}   \right\rangle \quad(\mbox{as } \nabla\cdot (e^{ix\cdot\xi}\phi^0_{m,\hat{\eta}}) =0).
\end{align}
Thus
\begin{equation}
 \chi_{\epsilon^{-1}\mathbb{T}^d}B^\epsilon_{m ,\hat{\eta}} (v \nabla p^\epsilon)  (\xi)   \rightarrow   \phi_{m,\hat{\eta}}^0\cdot \widehat{(v\nabla p)}(\xi)  \quad \mbox{ in }  L^2_{loc} (\mathbb{R}^d_\xi ) \mbox{ strongly.}\\[2mm]
\end{equation}
This property proved for $H^{-1}$ elements is analogous to Theorem \ref{thm2.1}. 

\paragraph{Summary so far : }  Combining the previous results, therefore, 
by taking the Bloch transformation $B_{m,\hat{\eta}}^{\epsilon}$ of the L.H.S. 
of \eqref{local} ($1\leq m \leq d-1$) and  multiplying by $\chi_{\epsilon^{-1}\mathbb{T}^d}$, 
we see that it converges to  
\begin{equation}\label{limit-lhs}
{(A^{*})}^{kl}_{\alpha\beta}\hat{\eta}_\alpha\hat{\eta}_\beta (\xi\cdot\hat{\eta})^2 (\phi_{m,\hat{\eta}}^{0})_l \widehat{(vu_k)}(\xi)
+ \phi_{m,\hat{\eta}}^0  \cdot \widehat{(v\nabla p)}(\xi) 
\quad\mbox{ in $L^2_{loc}(\mathbb{R}^d_{\xi})$  strongly. }
\end{equation}

\paragraph{Step 3. Limit of $B_{m,\hat{\eta}}^{\epsilon}$ applied to the R.H.S. of \eqref{local} : }
Applying
$B_{m,\hat{\eta}}^{\epsilon}$ to the R.H.S. of \eqref{local} ($1\leq m \leq d-1$ ), we obtain 
\begin{equation}\label{bt-rhs}
B_{m,\hat{\eta}}^{\epsilon}(vf ) (\xi ) +B_{m,\hat{\eta}}^{\epsilon} (g^\epsilon) (\xi)   
+B_{m,\hat{\eta}}^{\epsilon} (h^\epsilon)   (\xi). 
\end{equation}
We treat below each of these terms separately.
Passing to the limit in the first term is straightforward (cf. Corollary \ref{corol}) and we obtain
\begin{equation}
\chi_{\epsilon^{-1}\mathbb{T}^d}(\xi) B_{m,\hat{\eta}}^{\epsilon}(vf ) (\xi ) \rightarrow \phi_{m,\hat{\eta}}^{0}\cdot \widehat{(vf)} \mbox{ in } L^2_{loc}(\mathbb{R}^d_{\xi}) \mbox{ strongly.}
\end{equation}

\noindent{\bf Limit of $B^{\epsilon}_{m,\hat{\eta}}(g^{\epsilon})$ : } 
We pose  $\sigma^{\epsilon}= \mu^{\epsilon}\nabla u^{\epsilon}$ ($\sigma^\epsilon_{l\alpha} = \mu^\epsilon\frac{\partial u^\epsilon_l}{\partial x_\alpha}$) which is a bounded matrix in  $(L^2(\Omega))^{d \times d}$ and so  there exists a weakly convergent subsequence in $(L^2(\Omega))^{d \times d }$. Let $\sigma$ be its limit as well as its extension by zero outside $\Omega$.
Then via Theorem \ref{thm2.2}, 
\begin{equation}
\chi_{\epsilon^{-1}\mathbb{T}^d}(\xi)B_{m,\hat{\eta}}^{\epsilon}(\sigma^{\epsilon}_{l\alpha} \frac{\partial v}{\partial x_\alpha}) (\xi)  \rightharpoonup \widehat{(\sigma_{l\alpha}\frac{\partial v}{\partial x_\alpha})}(\xi) (\phi^0_{m, \hat{\eta}})_l  \quad\mbox{in }L^2_{loc}(\mathbb{R}_{\xi}^d) \mbox{ weakly.}
\end{equation}
Due to the strong convergence of $u^\epsilon$ in $L^2 (\mathbb{R}^d)^d,$ (cf. Corollary \ref{corol}) we have 
\begin{equation}
\chi_{\epsilon^{-1}\mathbb{T}^d}(\xi)B_{m,\hat{\eta}}^{\epsilon}(\mu^{\epsilon}\Delta v u^{\epsilon}) (\xi)  \rightharpoonup M_{\mathbb{T}^d}(\mu(y))\widehat{(\Delta v u)}(\xi) \cdot \phi^0_{m, 
\hat{\eta}} \quad\mbox{in }L^2_{loc}(\mathbb{R}_{\xi}^d)\mbox{ weakly}.\\[2mm] 
\end{equation}
Combining the above two convergence results and doing integration by parts, we obtain 
\begin{equation}
\chi_{\epsilon^{-1}\mathbb{T}^d}(\xi)B_{m,\hat{\eta}}^{\epsilon}(g^{\epsilon}) (\xi) \rightharpoonup -2\widehat{(\sigma \nabla v)}(\xi) \cdot \phi ^0_{m, \hat{\eta}} - M_{\mathbb{T}^d}(\mu(y))\widehat{(\Delta v u)}(\xi) \cdot \phi^0_{m, 
\hat{\eta}}\quad \mbox{ in } L^2_{loc} ( \mathbb{R}^d_\xi) \mbox{ weakly.}\\[3mm]
\end{equation}

\noindent{\bf Limit of $B^{\epsilon}_{m,\hat{\eta}}(h^{\epsilon})$ : } 
We decompose it into  two terms: 
\begin{equation}\begin{aligned}\label{decompo}
&B^{\epsilon}_{m,\hat{\eta}}(h^{\epsilon}) = - B_{m,\hat{\eta}}^{\epsilon}((\nabla \mu^{\epsilon}\cdot\nabla v) u^{\epsilon}) (\xi)\\
&\qquad = - \left\langle (\nabla \mu^{\epsilon}\cdot\nabla v) u^{\epsilon}, e^{i x\cdot\xi}\phi_{m,\hat{\eta}}^0 \right\rangle
  -  \left\langle (\nabla \mu^{\epsilon}\cdot\nabla v) u^{\epsilon}, e^{i x\cdot\xi}\epsilon(\xi\cdot\hat{\eta})\phi^{\prime}_m(\frac{x}{\epsilon};0) + \mathcal{O}((\epsilon(\xi\cdot\hat{\eta}))^2) \right\rangle. 
\end{aligned}\end{equation}
We start with the second term. By doing integration by parts, it becomes    
\begin{equation}\label{abv1}
(\xi\cdot\hat{\eta}) \int_{\mathbb{R}^d} e^{-i x\cdot\xi}\lb \mu^{\epsilon} 
\nabla _y\overline{\phi^\prime}_m(\frac{x}{\epsilon};0 ) u^\epsilon\rb \cdot \nabla v dx  + \mathcal{O}(\epsilon(\xi\cdot\hat{\eta})).
\end{equation}
Thanks to the strong convergence of $u^\epsilon$ in $L^2 (\mathbb{R}^d)^d$, 
the above quantity  converges in $L^2_{loc}(\mathbb{R}^d_\xi)$ strongly to
\begin{equation}\label{2nd-term}
  (\xi \cdot \hat{\eta})\int_{\mathbb{R}^d} e^{-ix\cdot\xi} \lb M_{\mathbb{T}^d}\lb \mu(y)\nabla_y\overline{\phi^\prime}_m(y;0)\rb u\rb\cdot \nabla v\ dx.
\end{equation}
Next, we  consider the first term of the R.H.S. of \eqref{decompo}.  After doing integration by parts, one has
\begin{equation}\label{above}
\int_{\mathbb{R}^d}e^{-i x\cdot\xi} \left[\mu^\epsilon\Delta v\ (u^{\epsilon}\cdot \phi_{m,\hat{\eta}}^0) + \lb(\mu^\epsilon\nabla u^\epsilon)\phi^0_{m,\hat{\eta}}\rb \cdot \nabla v  -  i\mu^\epsilon\lb(\phi_{m,\hat{\eta}}^0\otimes\xi)
u^\epsilon\rb\cdot\nabla v \right] dx.\\[2mm]  
\end{equation}
In a manner similar to the above arguments, the limit of \eqref{above} would be
\begin{equation}\label{1st-term}
\int_{\mathbb{R}^d}e^{-i x\cdot\xi} \left[M_{\mathbb{T}^d}(\mu(y))\Delta v\ (u\cdot \phi_{m,\hat{\eta}}^0) + \lb\sigma \phi_{m,\hat{\eta}}^0 \rb\cdot \nabla v -  iM_{\mathbb{T}^d}(\mu(y))\lb(\phi_{m,\hat{\eta}}^0\otimes\xi)u\rb\cdot\nabla v \right] dx.\\[2mm]
\end{equation}
Now combining \eqref{2nd-term} and \eqref{1st-term} and using the fact 
$$
\phi^{\prime}_m(y;0) -i\hat{\eta}_\beta\chi^l_\beta(y))(\phi_{m,\hat{\eta}}^0)_l
$$
is a constant vector of $\mathbb{C}^3$ independent of $y$, which in turn implies that 
$$ \nabla_y \phi^{\prime}_m(y;0) =i\hat{\eta}_\beta \nabla_y \chi^l_\beta(y)(\phi_{m,\hat{\eta}}^0)_l,
$$
we see that $\chi_{{\epsilon}^{-1}\mathbb{T}^{d}}B_{m,\hat{\eta}}^{\epsilon}(h^\epsilon)(\xi)$ converges 
strongly in $L^2_{loc}(\mathbb{R}^d_\xi)$ to
\begin{align}
&-i(\xi\cdot\hat{\eta})\left[ M_{\mathbb{T}^d} \lb \mu(y)  \hat{\eta}_\beta \nabla_y\chi^l_\beta(y)(\phi_{m,\hat{\eta}}^0)_l\rb  \right]_{k\alpha}\widehat{(\frac{\partial v}{\partial x_\alpha}u_k)}(\xi)\notag\\[2mm]
& +  M_{\mathbb{T}^d}(\mu(y))\widehat{(\Delta v\ u_k)}(\xi)(\phi_{m,\hat{\eta}}^0)_k  + \widehat{(\sigma_{l\beta}\frac{\partial v}{\partial x_\beta})}(\xi) (\phi_{m,\hat{\eta}}^0)_l - i M_{\mathbb{T}^d}(\mu(y))
(\phi_{m,\hat{\eta}}^0)_k\xi_\alpha \widehat{(\frac{\partial v}{\partial x_\alpha}u_k)}(\xi).
\end{align}

\paragraph{Step 4. Limit of $B_{m,\hat{\eta}}^{\epsilon}$ applied to \eqref{local} : } 
By equating the limiting identities that we have derived in the last two steps, we obtain 
\begin{align}\label{local-fourier}
&{(A^{*})}^{kl}_{\alpha\beta}\xi_\alpha\xi_\beta \widehat{(vu_k)}(\xi) (\phi_{m,\hat{\eta}}^{0})_l  + \widehat{(v\frac{\partial p}{\partial x_l})}(\xi)(\phi_{m,\hat{\eta}}^0)_l\notag\\[2mm]
&= \widehat{(vf_l)}(\xi)(\phi_{m,\hat{\eta}}^0)_l  -2\widehat{(\sigma_{ l\beta}\frac{\partial v}{\partial x_\beta})}(\xi)(\phi_{m,\hat{\eta}}^0)_l - M_{\mathbb{T}^d}(\mu(y))\widehat{(\Delta v\ u_k)}(\xi)(\phi_{m,\hat{\eta}}^0)_k \notag\\[2mm]
&\quad -i\left[ M_{\mathbb{T}^d} \lb \mu(y)\nabla_y\chi^l_\beta(y)\rb\right]_{k\alpha} (\phi_{m,\hat{\eta}}^0)_l \xi_\beta  \widehat{(\frac{\partial v}{\partial x_\alpha}u_k)}(\xi)+  M_{\mathbb{T}^d}(\mu(y))\widehat{(\Delta v\ u_k)}(\xi) (\phi_{m,\hat{\eta}}^0)_k\notag\\[2mm]
&\quad+ \widehat{(\sigma_{l\beta}\frac{\partial v}{\partial x_\beta})}(\xi) (\phi_{m,\hat{\eta}}^0)_l - i M_{\mathbb{T}^d}(\mu(y))\delta_{\alpha \beta}\delta_{l k}
(\phi_{m,\hat{\eta}}^0)_l\xi_\beta \widehat{(\frac{\partial v}{\partial x_\alpha}u_k)}(\xi).
\end{align}
The above equation has to be considered as the localized homogenized equation in the Fourier space. 
The conclusion of Theorem \ref{thm1.1} will follow as a consequence of this equation.

\paragraph{Step 5. Passage from Fourier space ($\xi)$ to physical space $(x)$ : }
We note that the L.H.S. and the R.H.S. of \eqref{local-fourier} can be written 
as $L(\xi)\cdot\phi_{m,\hat{\eta}}^0$ 
and $R(\xi)\cdot\phi_{m,\hat{\eta}}^0$, respectively, so that we have 
$$
\left[L(\xi)-R(\xi)\right]\cdot\phi_{m,\hat{\eta}}^0 =0\ \mbox{ for }m=1,..,(d-1).
$$
Observe that, the quantity $ \left[L(\xi)-R(\xi)\right]$ is independent of $m$. Varying $m=1,\ldots,(d-1)$ and using the fact $\xi \perp \phi_{m,\hat{\eta}}^0, \xi\in\mathbb{R}^d$  and $\{\phi^0_{1, \hat{\eta}} \cdots \phi^0_{d-1, \hat{\eta}}\}$  forms a basis of $\mathbb{R}^{d-1}$, we get
$$
\left[L(\xi)-R(\xi)\right] =c(\xi)\xi \quad\mbox{ for some scalar }c(\xi). $$
Therefore, for all test functions $w \in (L^2(\mathbb{R}^d))^d$ satisfying $\xi\cdot\hat{w}(\xi)=0$ (i.e. $\dive w =0 $ in $\mathbb{R}^d$ )
we also have
$$
\left[L(\xi)-R(\xi)\right]\cdot\hat{w}(\xi) =0. 
$$
Now by using the Plancherel's theorem, we have
\begin{equation}\label{Parseval}
\int_{\mathbb{R}^d} \mathcal{F}^{-1}\left[L(\xi)-R(\xi)\right](x)\cdot\overline{w}(x) dx =0,  \quad \forall w\in (L^2(\mathbb{R}^d))^d\mbox{ satisfying } \dive w =0
\end{equation}
where $\mathcal{F}^{-1}$ denotes the inverse Fourier transformation. 

We easily compute $I (x) = \mathcal{F}^{-1}\left[L(\xi)-R(\xi)\right](x)$ to obtain 
$$
I_l (x)  = \lb -{(A^{*})}^{kl}_{\alpha\beta}\frac{\partial^2 (vu_k)}{\partial x_\beta\partial x_\alpha} +  v\frac{\partial p}{\partial x_l}\rb- \lb vf_l - \sigma_{l, \beta} \frac{\partial v}{\partial x_\beta} - {(A^{*})}^{kl}_{\alpha\beta}\frac{\partial}{\partial x_\beta}(\frac{\partial v}{\partial x_\alpha} u_k ) \rb \quad\mbox{ in }\mathbb{R}^d,
$$
which simplifies in 
$$
I_{l} = \lb -{(A^{*})}^{kl}_{\alpha\beta}\frac{\partial^2 u_k}{\partial x_\beta\partial x_\alpha} +  \frac{\partial p}{\partial x_l} - f_l \rb v 
- \lb {(A^{*})}^{kl}_{\alpha\beta}\frac{\partial u_k}{\partial x_\alpha} - \sigma_{l, \beta}  \rb \frac{\partial v}{\partial x_\beta} \quad\mbox{in }\mathbb{R}^d.\\[2mm]
$$
We pose 
\begin{equation}\label{F1F2}
F^1_l =  \lb -{(A^{*})}^{kl}_{\alpha\beta}\frac{\partial^2 u_k}{\partial x_\beta\partial x_\alpha} +  \frac{\partial p}{\partial x_l} - f_l \rb \mbox{ and } F^2_{l\beta}= F^2_{\beta l}=
- \lb {(A^{*})}^{kl}_{\alpha\beta}\frac{\partial u_k}{\partial x_\alpha} - \sigma_{l, \beta}  \rb 
\end{equation}
to write $I_{l}$ in the form 
$$
 I_{l} =  F^1_l\ v + F^2_{l\beta}\ \frac{\partial v}{\partial x_\beta}.
$$
Using \eqref{Parseval}, it follows from de Rham's theorem that $I$ is a gradient and furthermore this is true whatever be $v\in\mathcal{D}(\Omega)$. This imposes restriction on $F^1,F^2$. In fact, we show using \eqref{Parseval} that $F^2_{l\beta}= q\delta_{l\beta}$ and $F^1 = \nabla q$  for some scalar $q\in L^2(\Omega)$ so that $I=v\nabla q + q\nabla v = \nabla (vq)$. \\

\paragraph{Step 5A. To show $F^2_{l\beta} = q\delta_{l\beta}$ :}
Let us choose $v = v_0 e^{inx\cdot \omega}$, where $\omega$ is a unit vector in $\mathbb{R}^d$ and $v_0 \in \mathcal{D}(\Omega)$ is fixed.
Next, we choose $w = \psi_{\zeta,\omega} \in (L^2(\mathbb{R}^d))^d$ where for any two constant perpendicular vectors $\zeta$ and $\omega$ in $\mathbb{R}^d$, $\psi_{\zeta,\omega}\in (L^2(\mathbb{R}^d))^d$ solves 
\begin{equation}\label{div-equation} 
\dive \psi_{\zeta,\omega} = 0 \mbox{ in }\mathbb{R}^d\ \mbox{ with }\psi_{\zeta,\omega} = \zeta e^{-inx\cdot\omega} \mbox{ in }\Omega, \mbox{ where } \zeta\perp\omega. 
\end{equation}  
The existence of such a function $\psi_{\zeta,\omega}$ can be shown as follows. 
Let $R_0>0$ be such that $\overline{\Omega} \subset B(0,R_0)$ and consider the 
following boundary value problem 
\begin{equation}\label{div-problem} 
\begin{array}{ll}
\dive \psi_{\zeta,\omega} = 0 & \mbox{ in }B(0,R_0)\smallsetminus \overline{\Omega}, \\
\psi_{\zeta,\omega} = 0 & \mbox{ on }\partial B(0,R_0) , \\
\psi_{\zeta,\omega} = \zeta e^{-inx\cdot\omega} & \mbox{ on }\partial\Omega. 
\end{array}
\end{equation}
There exists a solution of \eqref{div-problem} (see \cite[Page No. 24]{GR}) since 
the boundary data satisfies the required compatibility condition (recall that 
we assume $\zeta\cdot \omega = 0$) 
$$
\int_{\partial\Omega} \zeta e^{-inx\cdot\omega}\cdot \nu \ d\sigma = \int_\Omega (\zeta\cdot\omega) e^{-inx\cdot\omega}\ dx = 0 .
$$
Then extending $\psi_{\zeta,\omega}$ by $0$ outside $B(0,R_0)$ and by $\zeta e^{-inx\cdot\omega}$ in $\Omega$, clearly the extended function $\psi_{\zeta,\omega}$ solves \eqref{div-equation}. \\
\\
Now using these $v$ and $w$ in \eqref{Parseval}, we have
$$
 \int_{\Omega} F^1_l\zeta_l\ v_0\ dx + \int_\Omega F^2_{l\beta}\frac{\partial v_0}{\partial x_\beta}\zeta_l\ dx +  n\int_\Omega F^2_{l\beta}\omega_\beta\zeta_l\ v_0\ dx  = 0 
$$
and dividing by $n$ and letting $n\rightarrow \infty$ in the above relation,  we get
\begin{equation}\label{F^2}
 \int_{\Omega} ( F^2\ \omega\cdot \zeta ) v_0 = 0.
\end{equation}
As $v_0\in \mathcal{D}(\Omega)$ is arbitrary, \eqref{F^2} gives $F^2\ \omega\cdot \zeta = 0$ 
in $\Omega$. As $F^2$ is symmetric, and further using that $\omega, \zeta$ are arbitrary satisfying $\omega\cdot\zeta=0$, we conclude $F^2_{l\beta} = F^2_{\beta l} = q\delta_{l\beta}$ for some scalar function $q\in L^2(\Omega)$. This means that we have the relation :
\begin{equation}\label{flux-1}
 \sigma_{l\beta} =  q\delta_{l\beta} + {(A^{*})}^{kl}_{\alpha\beta}\frac{\partial u_k}{\partial x_\alpha} .
\end{equation}

\paragraph{Step 5B. To show $F^1 =\nabla q$ :} 
We choose  $v\in\mathcal{D}(\Omega)$ and $w= \psi_{e_k,0}$ with $\psi_{e_k,0}$ 
as in \eqref{div-equation} with $\zeta=e_k$ and $\omega=0$. Then using these $v$ and $w$ in \eqref{Parseval} and using the conclusion from \textbf{Step 5A}, we have
$$
\int_\Omega F^1_k v + q\frac{\partial v}{\partial x_k}\ dx = 0 \mbox{ for all }v\in\mathcal{D}(\Omega) ,
$$
which implies $(F^1_k -\frac{\partial q}{\partial x_k}) =0 $ for $k=1,..,d$ or, $F^1=\nabla q$. 

\paragraph{Step 5C.}
Using \textbf{Step 5A} and \textbf{Step 5B} in \eqref{F1F2}, and considering the relation $F^1 -\nabla F^2 = 0$ in $\Omega$, we get the macro balance equation :
\begin{equation}\label{balance}
 -\frac{\partial\sigma_{l\beta}}{\partial x_\beta} + \frac{\partial p}{\partial x_l} = f_l \ \mbox{ in } \Omega,\ \ l =1,\ldots,d.
\end{equation}

\paragraph{Step 5D.}
In this step, we prove that $q=0$ in $\Omega$ by using the divergence-free condition. 
Indeed, as $\nabla\cdot u^\epsilon = 0$ in $\Omega$, we have 
\begin{equation*}
\sigma^\epsilon_{ll} =  \mu^\epsilon\frac{\partial u^\epsilon_l}{\partial x_l} = 0  \quad\mbox{ in } \Omega. 
\end{equation*}
Passing to the limit $\epsilon \rightarrow 0$, we get 
\begin{equation*}
\sigma_{ll} = 0 \quad\mbox{ in } \Omega.
\end{equation*}
Using this relation in \eqref{flux-1} with $\beta=l$, we get 
\begin{equation}\label{qd}
(A^{*})^{k l}_{\alpha l}\frac{\partial u_k}{\partial x_\alpha} + q d = 0. 
\end{equation}
On the other hand, from \eqref{eq7} and \eqref{eq8} we have
\begin{equation*}
(A^{*})^{kl}_{\alpha l}  = \frac{1}{|\mathbb{T}^d|} \int\limits_{\mathbb{T}^d} \mu (y )   \nabla (\chi^{k}_\alpha +y_\alpha e_k) 
 :  \nabla (y_l e_l)\ dy  = \frac{1}{|\mathbb{T}^d|} \int\limits_{\mathbb{T}^d} \mu (y )   \frac{\partial }{\partial y_l}(\chi^{k}_\alpha +y_\alpha e_k)_l \ dy. \\[2mm]
\end{equation*}
Thus for fixed $k,\alpha=1,\ldots,d$ summing over $l$, since $\dive \chi^k_\alpha = 0$ in $Y$, 
we obtain 
\begin{equation}\label{ka}
(A^{*})^{kl}_{\alpha l}  = M_{\mathbb{T}^d}(\mu)\delta_{k\alpha} .
\end{equation}
Using \eqref{qd} and \eqref{ka}, as $\dive u =0$, we deduce
$$
q = -\frac{1}{d} (A^{*})^{k l}_{\alpha l}\frac{\partial u_k}{\partial x_\alpha} 
= - \frac{1}{d}M_{\mathbb{T}^d}(\mu) \delta_{k\alpha}\frac{\partial u_k}{\partial x_\alpha} = 0 .
$$
Finally, the macro constitutive law  follows as a consequence from \eqref{flux-1} : 
$$
\sigma_{l\beta} = {(A^{*})}^{kl}_{\alpha\beta}\frac{\partial u_k}{\partial x_\alpha}. 
$$

\paragraph{Step 5E.}
Since $q=0$, we deduce from \textbf{Step 5B} that $F^1=0$ and from \eqref{F1F2} 
we get the following  homogenized Stokes system satisfied by $u,p$ : 
\begin{equation}\begin{aligned}\label{hom-eq}
 -{(A^{*})}^{kl}_{\alpha\beta}\frac{\partial^2 u_k}{\partial x_\alpha \partial x_\beta}  +  \frac{\partial p}{\partial x_l} &= f_l  \mbox{ in }\Omega\ \mbox{ for } l =1,..,d. \\[1mm]
\dive u &= 0  \mbox{ in }\Omega\\[1mm]
u &= 0 \mbox{ on }\partial\Omega.
\end{aligned}\end{equation}
This completes the proof of Theorem \ref{thm1.1}. 

\bigskip

\paragraph{Acknowledgement :} This work has been carried out within a project supported by Indo - French Centre for Applied Maths -UMI, IFCAM. G. A. is a member of the DEFI project at INRIA Saclay Ile-de-France.

\bibliographystyle{plain}
\bibliography{Master_bibfile}

\def\cprime{$'$}
\begin{thebibliography}{10}

\bibitem{ACFO}
Gr{\'e}goire Allaire, Carlos Conca, Luis Friz, and Jaime~H. Ortega.
\newblock On {B}loch waves for the {S}tokes equations.
\newblock {\em Discrete Contin. Dyn. Syst. Ser. B}, 7(1):1--28 (electronic),
  2007.

\bibitem{BLP}
Alain Bensoussan, Jacques-Louis Lions, and George Papanicolaou.
\newblock {\em Asymptotic analysis for periodic structures}, volume~5 of {\em
  Studies in Mathematics and its Applications}.
\newblock North-Holland Publishing Co., Amsterdam-New York, 1978.

\bibitem{CHK}
Hi~Jun Choe and Hyunseok Kim.
\newblock Homogenization of the non-stationary {S}tokes equations with periodic
  viscosity.
\newblock {\em J. Korean Math. Soc.}, 46(5):1041--1069, 2009.

\bibitem{CPV}
C.~Conca, J.~Planchard, and M.~Vanninathan.
\newblock {\em Fluids and periodic structures}, volume~38 of {\em RAM: Research
  in Applied Mathematics}.
\newblock John Wiley \& Sons, Ltd., Chichester; Masson, Paris, 1995.

\bibitem{COV}
Carlos Conca, Rafael Orive, and Muthusamy Vanninathan.
\newblock Bloch approximation in homogenization and applications.
\newblock {\em SIAM J. Math. Anal.}, 33(5):1166--1198 (electronic), 2002.

\bibitem{CV}
Carlos Conca and Muthusamy Vanninathan.
\newblock Homogenization of periodic structures via {B}loch decomposition.
\newblock {\em SIAM J. Appl. Math.}, 57(6):1639--1659, 1997.

\bibitem{GV}
S.~Sivaji Ganesh and M.~Vanninathan.
\newblock Bloch wave homogenization of scalar elliptic operators.
\newblock {\em Asymptot. Anal.}, 39(1):15--44, 2004.

\bibitem{GR}
Vivette Girault and Pierre-Arnaud Raviart.
\newblock {\em Finite element methods for {N}avier-{S}tokes equations},
  volume~5 of {\em Springer Series in Computational Mathematics}.
\newblock Springer-Verlag, Berlin, 1986.

\bibitem{hornung}
U.~Hornung.
\newblock {\em Homogenization and Porous Media}.
\newblock Springer, New York, 1997.

\bibitem{MB}
R.~Morgan and I.~Babuska.
\newblock An approach for constructing families of homogenized equations for
  periodic media.
\newblock {\em SIAM J. Math. Anal.}, 2:1--33, 1991.

\bibitem{OZ}
Jaime~H. Ortega and Enrique Zuazua.
\newblock Generic simplicity of the eigenvalues of the {S}tokes system in two
  space dimensions.
\newblock {\em Adv. Differential Equations}, 6(8):987--1023, 2001.

\bibitem{RS}
M.~Reed and B.~Simon.
\newblock {\em Methods of modern mathematical physics}.
\newblock Academic Press, New York, 1978.

\bibitem{SP}
E.~Sanchez-Palencia.
\newblock {\em Non-Homogeneous Media and Vibration Theory}, volume 129 of {\em
  Springer Lecture Notes in Physics}.
\newblock Springer-Verlag, Berlin, 1980.

\bibitem{SaSy}
Fadil Santosa and William~W. Symes.
\newblock A dispersive effective medium for wave propagation in periodic
  composites.
\newblock {\em SIAM J. Appl. Math.}, 51(4):984--1005, 1991.

\bibitem{GV2}
Sista Sivaji~Ganesh and Muthusamy Vanninathan.
\newblock Bloch wave homogenization of linear elasticity system.
\newblock {\em ESAIM Control Optim. Calc. Var.}, 11(4):542--573 (electronic),
  2005.

\bibitem{wilcox}
C.~Wilcox.
\newblock Theory of bloch waves.
\newblock {\em J. Anal. Math.}, 33:146--167, 1978.

\end{thebibliography}
\end{document}